\documentclass[11pt,twoside]{amsart}

\usepackage{latexsym}
\usepackage{graphicx}
\usepackage{amsfonts,amsmath,amssymb,enumerate}
\usepackage{pb-diagram,pb-xy}
\usepackage{subfigure}
\usepackage[all]{xy}
\usepackage{mathrsfs}
\usepackage{graphics}
\usepackage{epsfig}
\usepackage{psfrag}
\usepackage{url}
\usepackage{fullpage}
\usepackage{mathrsfs}
\usepackage{epstopdf}

\DeclareMathOperator{\bl}{bl}

\newtheorem{theorem}{Theorem}
\newtheorem{lemma}[theorem]{Lemma}
\newtheorem{proposition}[theorem]{Proposition}

\newtheorem{corollary}[theorem]{Corollary}
\newtheorem{definition}[theorem]{Definition}
\newtheorem{conjecture}[theorem]{Conjecture}

\theoremstyle{remark}
\newtheorem{remark}{Remark}[theorem]

\newenvironment{example}{\noindent \textbf{Example}.}{\hfill $\square$}
\newenvironment{proofof}[1]{\noindent \textit{Proof of #1}.}{\hfill $\square$}

\newcommand{\NS}{\mathop{\rm NS}\nolimits}
\newcommand{\Num}{\mathop{\rm Num}\nolimits}
\newcommand{\Stab}{\mathop{\rm Stab}\nolimits}
\newcommand{\Coh}{\mathop{\rm Coh}\nolimits}
\newcommand{\Hom}{\mathop{\rm Hom}\nolimits}
\newcommand{\Amp}{\mathop{\rm Amp}\nolimits}

\numberwithin{equation}{section}
\newcommand{\Hilb}{\rm Hilb}

\newcommand{\rk}{\rm rk}

\def\Ext(#1,#2,#3){\mbox{Ext}^{#1}(#2,#3)}
\def\hom(#1,#2){\mbox{hom}(#1,#2)}
\def\ext(#1,#2,#3){\mbox{ext}^{#1}(#2,#3)}

\def\Z{\mathbb{Z}}
\def\R{\mathbb{R}}
\def\Q{\mathbb{Q}}
\def\C{\mathbb{C}}
\def\P{\mathbb{P}}

\def\T{\mathbb{T}}

\def\a'{\alpha}
\def\t'{\beta}
\def\g'{\gamma}
\def\l'{\lambda}
\def\pin(#1,#2){\left\langle #1,#2 \right\rangle}
\def\pti(#1,#2,#3){\left\langle #1,#2,#3 \right\rangle}
\def\suma(#1,#2,#3){\sum\limits_{#1=#2}^#3}

\begin{document}

\markboth{Cristian Martinez}
{Duality, Bridgeland wall-crossing and flips of secant varieties}

%%%%%%%%%%%%%%%%%%%%% Publisher's Area please ignore %%%%%%%%%%%%%%
%\catchline{}{}{}{}{}
%%%%%%%%%%%%%%%%%%%%%%%%%%%%%%%%%%%%%%%%%%%%%%%%%%%%%%%%%%%%%%%%%%%

\title{DUALITY, BRIDGELAND WALL-CROSSING AND FLIPS OF SECANT VARIETIES}

\author{CRISTIAN MARTINEZ}

\address{Department of Mathematics, University of California, Santa Barbara\\
South Hall, Room 6607 Santa Barbara, CA 93106-3080}

\email{martinez@math.ucsb.edu}
\urladdr{http://web.math.ucsb.edu/~martinez/}

%\author{OTHER D. AUTHOR}

%\address{Full affiliations \\
%and mailing addresses}

\begin{abstract}

Let $v_d(\mathbb{P}^2)\subset |\mathcal{O}_{\mathbb{P}^2}(d)|$ denote the $d$-uple Veronese surface. After studying some general aspects of the wall-crossing phenomena for stability conditions on surfaces, we are able to describe a sequence of flips of the secant varieties of $v_d(\mathbb{P}^2)$ by embedding the blow-up $\bl_{v_d(\mathbb{P}^2)}|\mathcal{O}_{\mathbb{P}^2}(d)|$ into a suitable moduli space of Bridgeland semistable objects on $\P^2$.
%We prove that for any smooth projective surface $X$, the functor $R\mathcal{H}om(\ \cdot\ ,\omega_{X})[1]$ induces an isomorphism between suitable Bridgeland moduli spaces. In particular, when $X=\P^2$ this reproves a result due to Maican for moduli spaces of Gieseker semistable plane sheaves supported on curves. We use the duality result to describe a sequence of flips of secant varieties of Veronese surfaces.
\end{abstract}



\maketitle

\setcounter{tocdepth}{1}
\tableofcontents

\section{Introduction}\label{intro}
Stability conditions on triangulated categories were introduced by Bridgeland \cite{Bstab}, who also constructed the first family of nontrivial examples for $K3$ surfaces \cite{B2}. In \cite{AB}, Arcara and Bertram extended these examples to an arbitrary smooth projective surface. \cite{AB} also studies the moduli spaces of Bridgeland semistable objects of a particular topological type on a simple K3 surface $(S,H)$. By analyzing how these moduli spaces change as the stability condition varies, the authors are able to describe a sequence of birational transformations of the blow-up of the complete linear series $|H|$ along $S$, flipping the secant varieties of $S$.

%\cite{AB} also provides us with a new perspective, by studying the variation of a particular family of stability conditions on a simple $K3$ surface $(X,H)$ for a particular topological type, the authors construct a sequence of birational transformations for the blow up of the complete linear series $|H|$ along $X$, flipping the ``secant" varieties of $X$.

One of the features of these stability conditions is their ``well behaved'' wall-crossing. This phenomenon was studied in \cite{ABCH} for the topological type $v=(1,0,-n)$ on $\P^2$, where it was indicated that varying the family of stability conditions introduced in \cite{AB} corresponds to running a directed Minimal Model Program (MMP) on the Hilbert scheme of $n$ points (regarded as the moduli space of Gieseker semistable sheaves of type $v$). The same statement was proven in \cite{BMW} for any primitive topological type. The correspondence between wall-crossing for Bridgeland stability conditions and MMP wall-crossing has been extensively studied by Coskun, Huizenga, and Wolf \cite{CH14,CHW14} to compute the nef and effective cones of moduli spaces of Gieseker semistable plane sheaves.

%In this paper, we want to study another ``classical" problem, flipping the secant varieties of Veronese surfaces, using Bridgeland's well behaved wall-crossing. 

In \cite{M5} and \cite{M6}, Maican constructs cohomological stratifications for the moduli spaces $N_{\P^2}(r,\chi)$ of Gieseker semistable sheaves on $\P^2$ with Hilbert polynomial $rm+\chi$ ($r=5,6$). Using those strata we can get exceptional loci for birational transformations of $N_{\P^2}(r,\chi)$, as it was done in \cite{BMW} for the moduli spaces $N(4,2)$ and $N(5,0)$. However, there is no bijective correspondence between the cohomological strata and the Bridgeland walls. It was shown in \cite{CHUNG13} that for the case of $N(6,1)$, a cohomological strata may be object of several contractions when running the MMP, giving rise to several Bridgeland walls. 

Nevertheless, when $\chi=0$ we can identify all rank-1 walls even when Maican-type stratifications are unknown. In this case, by restricting the Bridgeland wall-crossing on a suitable subvariety of a model of $N(d,0)$ ($d$ odd), and following the spirit of \cite{AB}, we construct a sequence of flips for the blow-up of the linear series $|\mathcal{O}(d-3)|$ along the Veronese surface, with the first of these flips coinciding with the one constructed by Vermeire in \cite{V1}.\\

\noindent \textbf{Theorem \ref{main3}.} \emph{Let $d\geq 5$ be an integer and let $\nu_{d-3}\colon \P^2\rightarrow \P(H^0(\mathcal{O}(d-3)))^{\vee}=\P^N$ be $(d-3)$-uple embedding. There exists a sequence of flips 
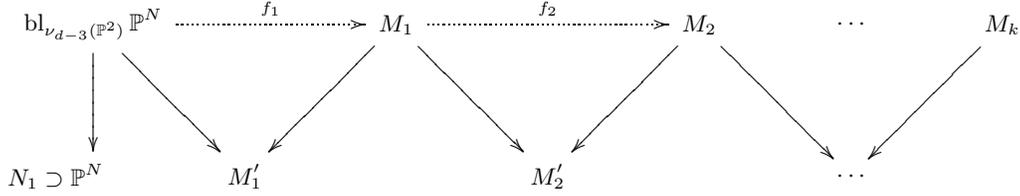
\footnotesize{
$$
\begin{diagram}
\node{\bl_{\nu_{d-3}(\P^2)}\P^N}\arrow{se}\arrow{s}\arrow[2]{e,t,..}{f_1}\node{}\node{M_1}\arrow[2]{e,t,..}{f_2}\arrow{sw}\arrow{se}\node{}\node{M_2}\arrow{sw}\arrow{se}\node{\cdots}\node{M_{k}}\arrow{sw}\\
\node{N_1\supset\P^N\hspace{1cm}}\node{M_1'}\node{}\node{M_2'}\node{}\node{\cdots}\node{}
\end{diagram} 
$$}
\normalsize
where $k=(d-1)/2$ for $d$ odd, and $k=\lfloor{(d-2)/4\rfloor}$ for $d$ even, the exceptional locus of $f_i$ is the strict transform of $Sec^{i}(\nu_{d-3}(\P^2))$, and $N_1$ is the first birational model appearing when running the MMP for $N(d,0)$ or $N(d,d/2)$ depending on whether $d$ is odd or even respectively.}\\

To obtain this sequence of flips we need to understand the flipping locus for the flips appearing when running the MMP for the Gieseker moduli spaces $N(d,0)$ for $d$ odd, and $N(d,d/2)$ for $d$ even. To do this, we will need a generalization of a result of Maican.
  
In \cite{MDUAL}, Maican proves that the map $\mathcal{F}\mapsto \mathscr{E}xt^{n-1}(\mathcal{F},\omega_{\P^n})$ induces an isomorphism between the moduli spaces $N_{\P^n}(r,\chi)$ and $N_{\P^n}(r,-\chi)$ of Gieseker semistable sheaves with Hilbert polynomials $P=rm+\chi$ and $P^D=rm-\chi$ respectively. The moduli spaces $N_X(r,\chi)$ were constructed by Simpson \cite{Simp} for any smooth projective surface via invariant theory and they were proven to be projective, so one could ask whether Maican's result extends to any surface. This was proven by Sacc\`a in her thesis \cite{GiuliaThesis}. We recover Sacc\`a's result as a corollary of a more general statement:\\
\\
\textbf{Theorem \ref{duality}.} \emph{The functor $R\mathcal{H}om(\cdot,\omega_X)[1]$ induces an isomorphism between the Bridgeland moduli spaces $M_{D,tH}(v)$ and $M_{-D+K_X,tH}(v^D)$ provided these moduli spaces exist and $Z_{D,tH}(v)$ belongs to the open upper half plane. }
\\

\textbf{Notation.} Other than specified we will use the following standard notation:
\begin{itemize}
\item $\mathfrak{R}(z)$, $\mathfrak{I}(z)$ denote the real and imaginary parts of the complex number $z$.
\item $D^b(X)$ is the bounded derived category of coherent sheaves on $X$.
\item $K(X)$ denotes the Grothendieck group of the triangulated category $D^b(X)$.
\item For $E,F\in K(X)$ we define the operator
$$
\chi(E,F)=\sum_{i\in\Z}(-1)^i \dim \Ext(i,E,F).
$$
When $E=\mathcal{O}$, we denote $\chi(E,F)$ by $\chi(F)$ and refer to it as the Euler characteristic of $F$.
\item When $X$ is a smooth projective complex surface, the Hizerburch-Riemann-Roch Theorem can be written as
$$
\chi(F)=ch_2(F)-ch_1(F)\frac{K_X}{2}+ch_0(F)\chi(\mathcal{O}),
$$ 
where $K_X$ is the canonical divisor of $X$.
\item $\mbox{Num}(X)$ denotes the group of cycles $A(X)$ up to numerical equivalence, and $\mbox{NS}(X)=\mbox{Num}^1(X)$ the Ner\'on-Severi group of divisors up to numerical equivalence. Also, $\mbox{Num}(X)_{\Q}$ and $\mbox{Num}(X)_{\R}$ denote the tensor products $\mbox{Num}(X)\otimes \Q$ and $\mbox{Num}(X)\otimes \R$, respectively.
%\item $\mathcal{N}(X)$ denotes the numerical Grothendieck group of $D^b(X)$, i.e., the quotient of $K(X)$ by the subgroup consisting of those $E\in K(X)$ such that $\chi(E,F)=0$ for all $F\in K(X)$, where
%$$
%\chi(E,F)=\sum_{i\in\Z}(-1)^i \dim \Ext(i,E,F).
%$$
\item We use $\mathcal{H}^i(\cdot)$ to denote the cohomology sheaves of an object in the derived category and $H^i(\cdot)$ for the cohomology groups of a sheaf.
\item We use $n\mathcal{F}$ to denote the direct sum $\mathcal{F}^{\oplus n}$.
\item For a smooth projective surface $X$, the topological type $v\in \mbox{Num}(X)_{\Q}$ of an object $E\in D^b(X)$ is its Chern character vector.
\item $M_H(v)$ denotes the moduli space of Gieseker semistable sheaves of topological type $v$ with respect to the polarization $H\in \mbox{Pic}(X)$.
\item For a stability condition $\sigma\in \mbox{Stab}(X)$, $M_{\sigma}(v)$ denotes the moduli space parametrizing $S$-equivalence classes of $\sigma$-semistable objects of topological type $v$ (if such space exists). For the stability conditions $\sigma_{\beta,\omega}$, $M_{\sigma_{\beta,\omega}}(v)$ is denoted simply by $M_{\beta,\omega}(v)$, and by $M_{s,t}(v)$ when $\beta=sH$ and $\omega=tH$.
\item We refer to an object $F$ fitting into an exact sequence $A\hookrightarrow F \twoheadrightarrow B$ as an extension.
\end{itemize}

	%) A SECTION HEADING
%Contributions to the {\it International Journal of Mathematics} 
%will be reproduced by photographing the author's 
%submitted typeset manuscript.  It is therefore essential that
%the manuscript be in its final form, and is an original computer
%printout because it will be printed directly without any
%editing. The manuscript should also be clean and unfolded. The
%copy should be evenly printed on a high resolution printer (300 
%dots/inch or higher). If typographical errors cannot be avoided,
%use cut and paste methods to correct them. Smudged copy, pencil
%or ink text corrections will not be accepted. Do not use 
%cellophane or transparent tape on the surface as this interferes
%with the picture taken by the publisher's camera.

%%%%%%%%%%%%%%%%%%%%%%%%%%%%%%%%%%%%%%%%%%%%%%%%%%

\section{Preliminaries}\label{prelim}
We assume familiarity with the concept of stability conditions introduced by Bridgeland \cite{Bstab}. We recall here the relevant theorems and definitions but for a more detailed presentation the unfamiliar reader is encouraged to consult Bridgeland's original papers \cite{Bstab,B2}, or the introduction to the topic by Huybrechts \cite{HSTAB}.   

Let $X$ be a smooth projective variety.
\begin{definition}A pre-stability condition on $X$ is a pair $\sigma=(Z,\mathcal{A})$ consisting of a linear function $Z:K(X)\rightarrow \C$ called the charge and the heart $\mathcal{A}$ of a bounded t-structure on $D^b(X)$, such that:
\begin{enumerate}
\item $\mathfrak{I}(Z(E))\geq 0$ for all $E\in \mathcal{A}$ and
\item If $\mathfrak{I}(Z(E))=0$ and $E\neq 0$ then $\mathfrak{R}(Z(E))<0$. 
\end{enumerate} 
\end{definition}

For every pre-stability condition one can define a slope function
$$
\mu_{\sigma}=\frac{-\mathfrak{R}(Z)}{\frak{I}(Z)}
$$
which gives us a notion of (semi)stability: an object $E\in\mathcal{A}$ is said to be $\sigma$-(semi)stable if for any inclusion $A\hookrightarrow E$ of objects in $\mathcal{A}$ one has
$$
\mu_{\sigma}(A)(\leq) < \mu_{\sigma}(E).
$$
\begin{definition}
A pre-stability condition $\sigma=(Z,\mathcal{A})$ is a stability condition if it has the Harder-Narasimhan property: 
\begin{itemize}
\item Every nonzero object $E\in\mathcal{A}$ admits a finite filtration in $\mathcal{A}$
$$
0\subset E_0\subset E_1\subset\cdots\subset E_n=E
$$
uniquely determined by the property that each quotient $F_i:=E_i/E_{i-1}$ is $\sigma$-semistable and $\mu_{\sigma}(F_1)>\mu_{\sigma}(F_2)>\cdots >\mu_{\sigma}(F_{n-1})$. 
\end{itemize}
\end{definition}
\begin{example}
If $X=C$ is a smooth projective curve then ordinary degree and rank of coherent sheaves give a stability condition on $\mathcal{A}=D^b(\mbox{Coh} (C))$:
$$
Z(\mathcal{F})=-\deg(\mathcal{F})+\sqrt{-1}\ \rk(\mathcal{F}).
$$
However, when $X$ is a surface this is not the case. One can still define a Mumford slope (with respect to some polarization $H$):
$$
\mu_H(E)=\frac{c_1(E)\cdot H}{\rk(E)},
$$
but this does not come from any stability condition on $\mbox{Coh}(X)$ since $c_1(\C_p)=\rk(\C_p)=0$. Nevertheless, it is true that every coherent sheaf $E$ has a filtration
$$
E_0\subset\cdots\subset E_n=E
$$
such that $E_0$ is the torsion subsheaf of $E$ and for every $i>0$,  the factors $E_i/E_{i-1}$ are semistable of decreasing slopes. 
\end{example}

Let $\sigma=(Z,\mathcal{A})$ be a stability condition on $X$. For any nonzero object $E\in\mathcal{A}$ one can write $Z(E)=|Z(E)|e^{\pi\sqrt{-1}\phi}$ for a unique $\phi\in(0,1]$. We say that $E$ has phase $\phi$. For every $\phi\in(0,1]$ we denote by $\mathcal{P}_{\sigma}(\phi)$ the subcategory consisting of $\sigma$-semistable objects of phase $\phi$. Inductively, one can define $\mathcal{P}_{\sigma}(\phi+1):=\mathcal{P}_{\sigma}(\phi)[1]$. For a bounded interval $I\subset \R$ we denote $\mathcal{P}_{\sigma}(I)$ the subcategory extension-generated by $\sigma$-semistable objects of phase in the interval $I$. For instance, $\mathcal{P}_{\sigma}(0,1]=\mathcal{A}$. 

One can define semistability in terms of phase just by declaring an object $E$ to be semistable if every subobject has smaller phase. This is equivalent to the definition using slopes since for an object $E\in\mathcal{A}$ of phase $\phi$ one has
$$
\mu_{\sigma}(E)=-\cot(\pi\phi).
$$ 
An easy but important consequence of the definition of stability is
\begin{proposition}[Schur's lemma] Let $\sigma=(Z,\mathcal{A})$ be a stability condition. 
\begin{enumerate}
\item If $E$ is $\sigma$-stable then $\Hom(E,E)=\C$.
\item If $A,B$ are different $\sigma$-stable objects of the same phase then $\Hom(A,B)=0$.
\item If $A\in\mathcal{P}_{\sigma}(\phi_1)$, $B\in\mathcal{P}_{\sigma}(\phi_2)$ with $\phi_1>\phi_2$ then $\Hom(A,B)=0$.
\end{enumerate}
\end{proposition}

Let $E\in\mathcal{P}_{\sigma}(\phi)$. A finite Jordan-H\"older filtration of $E$ is a chain
$$
0=E_0\subset E_1\subset\cdots\subset E_n=E
$$
such that the factors $E_i/E_{i-1}$ are stable of phase $\phi$. In general, finite Jordan-H\"older filtrations do not always exist, and even when they exist, they are not necessarily unique. However, the stable factors are always unique up to a permutation. 
\begin{definition}
A stability condition is called locally finite if there is some $\delta>0$ such that each quasi-abelian category $\mathcal{P}_{\sigma}(\phi-\delta,\phi+\delta)$ is of finite length. For a locally finite stability condition the categories 
$\mathcal{P}_{\sigma}(\phi)$ have finite length. In particular, every semistable object has a finite Jordan-H\"older filtration.
\end{definition}
\begin{definition} Let $\sigma$ be a locally finite stability condition. Two objects $A,B\in\mathcal{P}_{\sigma}(\phi)$ are called $S$-equivalent with respect to $\sigma$ if they have isomorphic Jordan-H\"older $\sigma$-stable factors (up to a reordering).
\end{definition}
\begin{definition} We say that a pre-stability condition $\sigma=(Z,\mathcal{A})$ is numerical if $Z\colon K(X)\rightarrow \C$ factors through the Chern character map $ch\colon K(X)\rightarrow \mbox{Num}(X)_{\Q}$. We continue denoting by $Z$ the corresponding homomorphism $\mbox{Num}(X)_{\Q}\rightarrow \C$.
\end{definition}
\begin{definition}
Fix a norm $\|\ \|$ on $\Num(X)_{\R}$. A numerical pre-stability condition $\sigma=(Z,\mathcal{A})$ is said to satisfy the support property if there exists a constant $C>0$ such that for all $\sigma$-semistable objects $0\neq E\in D^b(X)$, we have
$$
\| ch(E)\|\leq C|Z(E)|.
$$
\end{definition}

It follows from \cite[Proposition~B4]{BMlpp} and \cite[Lemma~4.5]{B2} that numerical stability conditions satisfying the support property are locally finite. We denote by $\Stab(X)$ the set of numerical stability conditions satisfying the support property. We can now state Bridgeland's deformation result:

\begin{theorem}[{\cite[Theorem 1.2]{Bstab}}]
There is a natural topology on $\Stab(X)$ such that the forgetful map
$$
\mathcal{Z}\colon \Stab(X)\rightarrow {\Hom}(\Num(X),\C),
$$
sending $\sigma=(Z,\mathcal{A})$ to $Z$, is a local homeomorphism. In particular, every connected component of $\Stab(X)$ is a complex manifold.
\end{theorem}
The following proposition due to Bridgeland expresses an important property of $\Stab(X)$:
\begin{proposition}[{\cite[Proposition 9.3]{B2}}]\label{bridwalls} Let $V\subset \Stab(X)$ be a connected component, let $K\subset V$ be a compact subset, and let $v\in\Num(X)_{\Q}$. Then there is a finite collection $\{W_{\gamma}\colon \gamma\in F\}$ of (real) codimension 1 submanifolds  of $V$ (not necessarily closed), such that every connected component 
$$
C\subset K\setminus \bigcup_{\gamma\in F}W_{\gamma}
$$  
has the property that if an object $E\in D^b(X)$ with $ch(E)=v$ is semistable for some stability condition in $C$, then it is semistable for all stability conditions in $C$. 
\end{proposition}
\begin{definition}
For a given topological type $v\in\Num(X)_{\Q}$, we refer to the submanifolds $W_{\gamma}$ of Proposition \ref{bridwalls} as walls of type $v$, and to every connected component $C$ as chamber of type $v$.
\end{definition}

\begin{proposition}[{\cite[Proposition 2.2.2]{BMT}}]\label{BMTprop} Given $E\in D^b(X)$, the set of $\sigma\in\Stab(X)$ for which $E$ is $\sigma$-stable is an open subset of $\Stab(X)$. Further, on the open part of every chamber in the wall and chamber decomposition of $\Stab(X)$, the Harder-Narasimhan filtration of $E$ is constant.
\end{proposition}
\begin{remark}\label{new1} There is a particular instance of Proposition \ref{BMTprop} that we will encounter throughout the paper. Assume that $\{\sigma_t=(Z_t,\mathcal{A})\}_{t\in T}$ is a continuous 1-parameter family of stability conditions with a fixed heart $\mathcal{A}$. Assume further that there is an object $E\in\mathcal{A}$ that is $\sigma_{t_1}$-stable and $\sigma_{t_2}$-unstable for some $t_1$ and $t_2$. Then the path $\{\sigma_t\}_{t\in T}$ can not be fully contained in a chamber of type $ch(E)$, and there is a value $t_0$ between $t_1$ and $t_2$ that is the intersection of a wall of type $ch(E)$ with the family $\{\sigma_t\}_{t\in T}$. By Proposition \ref{BMTprop} there is a subobject $A\hookrightarrow E$ in $\mathcal{A}$ (the first Harder-Narasimhan factor of $E$ in the nearby chamber where $E$ is unstable) such that $\mu_{\sigma_{t_0}}(A)=\mu_{\sigma_{t_0}}(E)$.    
\end{remark}
\begin{definition}\cite{AP}. Let $S$ be a scheme of finite type over $\C$. A family of objects in $\mathcal{A}$ parametrized by $S$ is an object $E\in D^b(X\times S)$ such that for every closed point $s\in S$ we have
$$
L i_s^*(E)\in \mathcal{A}.
$$
\end{definition}

From now on, $X$ is a smooth projective complex surface. We start by reviewing the examples of stability conditions constructed by Bridgeland \cite{B2} for the $K3$ case, and generalized by Arcara-Bertram \cite{AB} for arbitrary $X$. 

Fix a class $\omega\in \mbox{Amp}(X)$. One defines, for every $s\in\R$, the full subcategories of $\Coh(X)$:
\begin{itemize}
\item $\mathcal{Q}_s=\langle E\in \mbox{Coh}(X)\colon E \  \text{is}\   \mu_{\omega}\text{-semistable}\  \text{and}\  \mu_{\omega}(E)>s\rangle$,
\item $\mathcal{F}_s=\langle E\in \mbox{Coh}(X)\colon E\    \text{is}\   \mu_{\omega}\text{-semistable}\  \text{and}\  \mu_{\omega}(E)\leq s\rangle$,
\end{itemize}  
where for a collection of objects $\mathcal{B}\subset D^b(X)$, $\langle \mathcal{B}\rangle$ denotes the smallest subcategory containing $\mathcal{B}$ and that is closed under extensions. We follow the convention that a torsion sheaf is semistable of slope $+\infty$.

The subcategories $\mathcal{Q}_s,\mathcal{F}_s$ are full and $(\mathcal{Q}_s,\mathcal{F}_s)$ is a torsion pair, i.e.,
\begin{itemize}
\item $\Hom(Q,F)=0$ for all $Q\in\mathcal{Q}_s$, $F\in\mathcal{F}_s$.
\item Every coherent sheaf $E$ fits into an exact sequence
$$
0\rightarrow Q\rightarrow E\rightarrow F\rightarrow 0
$$
for some $Q\in\mathcal{Q}_s$, $F\in\mathcal{F}_s$. This short exact sequence is unique up to isomorphisms of extensions.
\end{itemize}
By general theory of torsion pairs \cite[Proposition 2.1]{HRS}, we know that the extension closure $\langle \mathcal{Q}_s,\mathcal{F}_s[1]\rangle$ is the heart of a bounded $t$-structure. More precisely, it is the full subcategory
$$
\mathcal{A}_s=\{E\in D^b(X)\colon H^{-1}(E)\in \mathcal{F}_s,\ H^0(E)\in\mathcal{Q}_s\ \text{and}\ H^{i}(E)=0,\ i\neq-1,0\}.
$$
\begin{theorem}[{\cite[Proposition~7.1]{B2}, \cite[Corollary~2.1]{AB}}]\label{stabilityfunction} 
Let $\beta,t\omega\in N^1(X)_{\Q}$ with $\omega$ ample. Then,
$$
Z_{\beta,t\omega}(E)=-\int_X e^{-\beta-\sqrt{-1}t\omega}ch(E)
$$
is the charge of a locally finite stability condition on $\mathcal{A}_{\beta\cdot \omega}$. We denote this stability condition by $\sigma_{\beta,t\omega}$.
\end{theorem}
\begin{remark}
It follows directly from the Hodge index theorem and the Bogomolov-Gieseker inequality that $\sigma_{\beta,t\omega}$ is a pre-stability condition for every $\beta,\omega\in \NS(X)_{\R}$ and $t>0$. However, producing Harder-Narasimhan filtrations for arbitrary classes is more complicated. If for every $\beta,t\omega\in \NS(X)_{\Q}$ the stability conditions $\sigma_{\beta,t\omega}$  satisfy the support property, then Bridgeland's deformation result guarantees the existence of a lift of $Z_{\beta,t\omega}\in \Hom(\mbox{Num}(X),\C)$ to $\Stab(X)$. Because for every $\beta,t\omega\in \NS(X)_{\Q}$ the skyscraper sheaves $\C_x$ are $\sigma_{\beta,t\omega}$-stable, then for arbitrary real classes $\beta,\omega\in \NS(X)_{\R}$ and $t>0$, there should be a lift of $Z_{\beta,t\omega}$ satisfying the same property (within the same chamber of type $ch(\C_x)$ of nearby rational classes). Then \cite[Lemma 10.1(d)]{B2} shows that the only possible lift of $Z_{\beta,t\omega}$ for which all the skyscraper sheaves $\C_x$ are stable is $\sigma_{\beta,t\omega}$. The proof that on an arbitrary smooth projective surface $X$, the stability conditions $\sigma_{\beta,t\omega}$ satisfy the support property is due to Toda  \cite[Proposition~3.13]{Tod}.
\end{remark}

Assume that there is a short exact sequence $0\rightarrow A\rightarrow E\rightarrow B\rightarrow 0$ in $\mathcal{A}_{\beta\cdot\omega}$ with $ch_0(A)=ch_0(E)=ch_0(B)=0$, then it follows from the definition of $Z_{\beta,t\omega}$ that $\mu_{\sigma_{\beta,t\omega}}(A)=\mu_{\sigma_{\beta,t\omega}}(E)$ if and only if
\begin{equation}\label{torsion subsheaf}
 \frac{\chi(A)}{ch_1(A)\cdot\omega}+\frac{ch_1(A)\cdot(\frac{K_X}{2}-\beta)}{ch_1(A)\cdot\omega}=\frac{\chi(E)}{ch_1(E)\cdot\omega}+\frac{ch_1(E)\cdot(\frac{K_X}{2}-\beta)}{ch_1(E)\cdot\omega}. 
\end{equation}
Because this condition is independent of $t$, then any such $E$ is not $\sigma_{\beta,t\omega}$-stable for any value of $t$. However, it is possible for $E$ to be $\sigma_{\beta,t\omega}$-semistable for some values of $t$ and $\sigma_{\beta,t\omega}$-unstable for others. This motivates the following definition:
\begin{definition} An object $E\in \mathcal{A}_{\beta\cdot\omega}$ of Chern character vector $ch(E)=(0,ch_1,ch_2)$ is said to be $\sigma_{\beta,t\omega}$-pseudo-stable if $E$ is $\sigma_{\beta,t\omega}$-semistable, and for any subobject $A\hookrightarrow E$ in $\mathcal{A}_{\beta\cdot\omega}$ we have
$$
\mu_{\sigma_{\beta,t\omega}}(A)=\mu_{\sigma_{\beta,t\omega}}(E)\ \Rightarrow\ ch_0(A)=0.
$$
\end{definition}
%%%%%%%%%%%%%%%%
\subsection{Stability conditions on the projective plane}
We now concentrate in the case $X=\P^2$. In this case (Picard number 1), one can think of $ch(E)$ as a vector with numerical entries. Choosing $\omega=H$, the hyperplane class, and $\beta=sH$, the central charge takes the form
$$
Z_{s,t}(ch_0,ch_1,ch_2)=(-ch_2+ch_1s-\frac{ch_0}{2}(s^2-t^2))+\sqrt{-1}\ (ch_1-ch_0s)t.
$$
One of the most important results in \cite{ABCH} is the following:
%\begin{theorem}[\cite{AP}] For every $(s,t)$ in the upper-half plane, the moduli stack $\mathcal{M}_{s,t}(ch_0,ch_1,ch_2)$ classifying $S$-equivalence classes of families of $Z_{s,t}$-semistable objects in $\mathcal{A}_s$ is of finite type. The moduli stack of stable objects is separated and if all semistable objects are stable then the moduli stack is proper.
%\end{theorem}
\begin{theorem}[{\cite[Proposition~8.1]{ABCH}}]\label{quiver} There are projective coarse moduli spaces $M_{s,t}(v)$ classifying $S$-equivalence classes of families of $\sigma_{s,t}$-semistable objects in $\mathcal{A}_s$ of topological type $v$.  
\end{theorem}
The idea is to identify $\sigma_{s,t}$-stability with quiver stability. Let $k\in\Z$ and consider the extension closure
$$
\mathcal{A}(k)=\langle \mathcal{O}(k-2)[2],\mathcal{O}(k-1)[1],\mathcal{O}(k)\rangle.
$$ 
An element of $\mathcal{A}(k)$ is a complex
$$
\C^{n_0}\otimes\mathcal{O}(k-2)\rightarrow \C^{n_1}\otimes\mathcal{O}(k-1)\rightarrow \C^{n_2}\otimes\mathcal{O}(k)
$$
with dimension vector $\mathfrak{n}=(n_0,n_1,n_2)$. Let $\mathfrak{a}$ be a vector orthogonal to $\mathfrak{n}$. An object of dimension vector $\mathfrak{n}$ is said to be quiver (semi)stable with respect to $\mathfrak{a}$ if for any subcomplex in $\mathcal{A}(k)$ of dimension vector $\mathfrak{n}'$ one has $\mathfrak{n}'\cdot \mathfrak{a}(\geq)>0$.

Moduli spaces of quiver semistable complexes of fixed dimension vector with respect to a fixed $\mathfrak{a}$ have a construction via GIT given in \cite{K}.  The change from Chern classes to dimension vectors is given by the matrix
$$
B_k:=\begin{bmatrix}
\frac{k(k-1)}{2} & \frac{-(2k-1)}{2} & 1\\
k(k-2) & -(2k-2) & 2\\
\frac{(k-1)(k-2)}{2} & \frac{-(2k-3)}{2}& 1
\end{bmatrix}.
$$
Fix a Chern character $v$ so that $B_kv$ is a vector of nonnegative entries. It is shown in \cite[Proposition 7.3]{ABCH} that for every $(s,t)$ in the region
$$
(s-(k-1))^2+t^2<1,
$$ 
there exists a choice of an orthogonal vector $\mathfrak{a}_{s,t}$ such that the moduli spaces of $\sigma_{s,t}$-semistable objects of Chern character $v$ are isomorphic to moduli spaces of semistable objects in $\mathcal{A}(k)$ of dimension vector $B_kv$ with respect to $\mathfrak{a}_{s,t}$. The key point is to observe that each wall and chamber of type $v$ in the $(s,t)$-plane intersects at least one of the regions above and so, by Proposition \ref{bridwalls}, the moduli spaces $M_{s,t}(v)$ are projective and may be constructed by GIT.

Notice that in general the Gieseker moduli spaces of 1-dimensional plane sheaves are not smooth (although its singular locus has high codimension). Because torsion sheaves are objects of $\mathcal{A}_s$ then any short exact sequence of sheaves $0\rightarrow A\rightarrow E\rightarrow B\rightarrow 0$, with $E$ a torsion sheaf, would also be exact in $\mathcal{A}_s$. Moreover, it follows from equation \eqref{torsion subsheaf} that a sheaf destabilizing $E$ with respect to Gieseker stability would also destabilize $E$ with respect to Bridgeland stability. Then on any chamber of type $v=(0,ch_1,ch_2)$ in the $(s,t)$-plane one has $\sigma_{s,t}$-semistable $=$ $\sigma_{s,t}$-pseudo-stable for objects of topological type $v$. 

%Also for $t\gg 0$ a sheaf is $\sigma_{s,t}$-pseudo-stable if and only if it is Gieseker semistable and these are all the $\sigma_{s,t}$-semistable objects (see proof of Corollary \ref{duality2}). Thus, Theorem 1.1 in \cite{BMW} holds for 1-dimensional plane sheaves, with no restrictions on the topological type, when replacing stable by pseudo-stable. 

%%%%%%%%%%%%%%%%%%%%%%%%%%%%%%%%%%%%%%%%%%%%%%%%%%

\section{Wall and Chamber Structure}\label{wall+chamber}
The results in this section seem to be known to the experts but we decided to include here some proofs for the sake of completeness.

In \cite{ABCH} for the case of the stability conditions $\sigma_{s,t}$ on $\P^2$, the authors describe what new objects become stable after crossing a wall. Similar reasoning can be applied to study the wall-crossing for the 1-dimensional family of stability conditions $\{\sigma_{D,tH}\}_{t>0}$ for fixed $D, H$ with $H$ ample, on an arbitrary smooth projective surface. Following Remark \ref{new1},  assume that a wall at $t=t_0$ is produced by a destabilizing sequence
$$
0\rightarrow A\rightarrow E^+\rightarrow B\rightarrow 0
$$
and assume furthermore that $A$ and $B$ are stable at the wall and $E^+$ is stable above the wall ($t>t_0$). Then the destabilizing sequence is a Jordan-H\"older filtration for the semistable object $E^+$ at $t_0$. Crossing the wall will produce semistable objects that are $S$-equivalent to $E^+$ at the wall, i.e., the new objects have $A$ and $B$ as stable factors and so they are extensions of the form 
$$
0\rightarrow B\rightarrow E^-\rightarrow A\rightarrow 0.
$$
But even more is true, 
\begin{proposition}\label{MMP1}
Assume that $\mu_{D,t_0H}(A)=\mu_{D,t_0H}(B)$ for some objects $A,B\in \mathcal{A}_{DH}$ and that there is $\epsilon>0$ such that $A$ and $B$ are $\sigma_{D,tH}$-stable with $\mu_{D,tH}(A)<\mu_{D,tH}(B)$ for $t_0\leq t<t_0+\epsilon$. If $\Ext(1,B,A)\neq0$ in $\mathcal{A}_{DH}$ then there exists $\delta>0$ such that every non-split extension
$$
0\rightarrow A\rightarrow E\rightarrow B\rightarrow 0
$$
is $\sigma_{D,tH}$-stable (or $\sigma_{D,tH}$-pseudo-stable when $ch_0(E)=0$) for all $t_0<t<t_0+\delta$.
\end{proposition}
\begin{proof}
Assume for the moment that $ch_0(E)\neq 0$. Let $0<\delta\leq \epsilon$ such that there are no walls for $E$ between $t_0$ and $t_0+\delta$ (this is possible because the walls are locally finite). It is enough to prove that there is no stable subobject $E'\hookrightarrow E$, $\sigma_{D,tH}$-destabilizing $E$ for $t_0<t<t_0+\delta$. If there were such $E'$ then at the wall $t_0$, $E'$ is semistable and $\mu_{D,t_0H}(E')=\mu_{D,t_0H}(E)$, otherwise it would destabilize $E$ at $t_0$ since in any case $\mu_{D,t_0H}(E')\geq \mu_{D,t_0H}(E)$. 

Because $B$ is stable at the wall $t_0$ and $\mu_{D,t_0H}(E')=\mu_{D,t_0H}(B)$, then the composition $E'\rightarrow E\rightarrow B$ is either surjective or zero. If it were the zero map, then we would have an inclusion $E'\hookrightarrow A$ in which case $\mu(E')<\mu(A)<\mu(E)$ above the wall $t_0$. Let $K$ be the  kernel of $E'\rightarrow B$, then there is an inclusion $K\hookrightarrow A$. Since the slopes of $K$ and $A$ are equal at $t_0$, then either $K=0$ in which case the sequence $A\rightarrow E\rightarrow B$ splits, or $K=A$ and therefore $E'=E$.

When $ch_0(E)=0$ there is the possibility that $E$ has a subobject $E'\hookrightarrow E$ with $\mu_{D,tH}(E')=\mu_{D,tH}(E)$ for all $t>0$, making it impossible for $E$ to ever be stable. Fortunately, this only happens if $ch_0(E')=0$, thus by further assuming $ch_0(E')\neq 0$ in the argument above we obtain $E$ $\sigma_{D,tH}$-psudo-stable. 
\end{proof}
\begin{remark} The conclusion of Proposition \ref{MMP1} also holds below the wall, i.e., if $A$ and $B$ are $\sigma_{D,tH}$-stable with $\mu_{D,tH}(A)<\mu_{D,tH}(B)$ for $t_0-\epsilon < t\leq t_0$, then there is $\delta>0$ such that every nonsplit extension $0\rightarrow A\rightarrow E\rightarrow B\rightarrow 0$ is $\sigma_{D,tH}$-stable (or $\sigma_{D,tH}$-pseudo-stable when $ch_0(E)=0$) for all $t_0-\delta<t<t_0$.
\end{remark}

Moreover, the more general result holds:
\begin{proposition}\label{MMP3}
Let $E$ be an object in $\mathcal{A}_{DH}$ which is strictly $\sigma_{D,t_0H}$-semistable for some $t_0>0$. Assume that $E$ has a Jordan-H\"older filtration at the wall determined by $t_0$ that becomes the Harder-Narasimhan filtration of $E$ on one of the chambers determined by $t_0$, then $E$ is $\sigma_{D,tH}$-stable (or $\sigma_{D,tH}$-pseudo-stable when $ch_0(E)=0$) on the other chamber. 
\end{proposition}
\begin{proof}
Without loss of generality we can assume that $E$ is $\sigma_{D,tH}$-unstable for $t<t_0$. Assume that at the wall determined by $t_0$, $E$ has a Jordan-H\"older filtration 
$$
0=E_0\subset E_1\subset \cdots \subset E_{n-1}\subset E_n=E
$$
such that for $t$ sufficiently near $t_0$ and above the wall $F_i=E_i/E_{i-1}$ is $\sigma_{D,tH}$-stable and the sequence $\mu_{D,tH}(F_i)$ is strictly increasing. Then by applying Proposition \ref{MMP1} to the exact sequences
$$
0\rightarrow E_{i-1}\rightarrow E_i\rightarrow F_i\rightarrow 0
$$
we conclude that there exist $\delta>0$ such that each $E_i$ is $\sigma_{D,tH}$-stable for all $t\in (t_0,t_0+\delta)$. In particular $E_n=E$ is $\sigma_{D,tH}$-stable for every $t$ in this interval (or $\sigma_{D,tH}$-pseudo-stable if $ch_0(E)=0$).
\end{proof}
\begin{proposition}\label{MMP2} Let us assume that the length of a Jordan-H\"older filtration for $E$ (and so of any) is 2 at a wall determined by $t_0>0$, and that $E$ fits into a diagram
$$
\begin{diagram}\dgARROWLENGTH=1.5em
\node{}\node{}\node{B'}\arrow{s,J}\node{}\\
\node{A}\arrow{e,J}\node{E}\arrow{e,A}\node{B}\arrow{s,A}\\
\node{}\node{}\node{B''}\\
\end{diagram}
$$
where $A,B'$ and $B''$ are the $\sigma_{D,t_0H}$-stable factors of $E$. Assume that there is $\epsilon>0$ such that for all $t\in(t_0,t_0+\epsilon)$ $A,B',B''$ are $\sigma_{D,tH}$-stable and $\mu_{D,tH}(A)<\mu_{D,tH}(B')<\mu_{D,tH}(B'')$. Then there exists $\delta>0$ such that objects $\tilde{E}$ that are extensions of the form
$$
\begin{diagram}\dgARROWLENGTH=1.5em
\node{}\node{}\node{}\node{B'}\node{}\\
\node{0}\arrow{e}\node{A}\arrow{e}\node{\tilde{E}}\arrow{e}\node{\tilde{B}}\arrow{n,A}\arrow{e}\node{0}\\
\node{}\node{}\node{}\node{B''}\arrow{n,J}\node{}
\end{diagram}
$$
can not be stable for any  $t_0<t<t_0+\delta$.
\end{proposition}
\begin{proof}
By Proposition \ref{MMP1} there exists $\delta>0$ such that all extensions 
$$
0\rightarrow A\rightarrow E\rightarrow B\rightarrow 0
$$
are $\sigma_{D,tH}$-stable for $t_0<t<t_0+\delta$. Taking $\delta<\epsilon$ we have $\Hom(B',A)=0$ since $\mu_{D,tH}(A)<\mu_{D,tH}(B')$, and therefore we get an inclusion
$$
\Ext(1,B'',A)\hookrightarrow\Ext(1,B,A).
$$
The image of every nonzero element corresponds to a non split extension which is stable by Proposition \ref{MMP3}. Such extensions admit an injective morphism $B'\hookrightarrow E$ that can be visualized in the diagram
$$
\begin{diagram}\dgARROWLENGTH=1.5em
\node{}\node{}\node{}\node{0}\arrow{s}\node{}\\
\node{}\node{}\node{}\node{B'}\arrow{sw,t}{\exists}\arrow{s}\node{}\\
\node{0}\arrow{e}\node{A}\arrow{s,=}\arrow{e}\node{E}\arrow{s}\arrow{e}\node{B}\arrow{s}\arrow{e}\node{0}\\
\node{0}\arrow{e}\node{A}\arrow{e}\node{G}\arrow{e}\node{B''}\arrow{e}\arrow{s}\node{0}\\
\node{}\node{}\node{}\node{0}\node{}\\
\end{diagram}
$$

Stability of $E$ implies $\mu_{D,tH}(B')<\mu_{D,tH}(E)=\mu_{D,tH}(\tilde{E})$ and so $B'$ destabilizes $\tilde{E}$ via the composition $\tilde{E}\twoheadrightarrow \tilde{B}\twoheadrightarrow B'$. Thus the only other possibility is $\Ext(1,B'',A)=0$  which gives a surjective map $\Ext(1,B',A)\twoheadrightarrow \Ext(1,\tilde{B},A)$ implying that $\tilde{E}$ is a pullback of an extension of $B'$ by $A$. As before, there is an injective map $B''\hookrightarrow \tilde{E}$
$$
\begin{diagram}\dgARROWLENGTH=1.5em
\node{}\node{}\node{}\node{0}\arrow{s}\node{}\\
\node{}\node{}\node{}\node{B''}\arrow{sw,t}{\exists}\arrow{s}\node{}\\
\node{0}\arrow{e}\node{A}\arrow{s,=}\arrow{e}\node{\tilde{E}}\arrow{s}\arrow{e}\node{\tilde{B}}\arrow{s}\arrow{e}\node{0}\\
\node{0}\arrow{e}\node{A}\arrow{e}\node{\tilde{G}}\arrow{e}\node{B'}\arrow{e}\arrow{s}\node{0}\\
\node{}\node{}\node{}\node{0}\node{}\\
\end{diagram}
$$

Again the stability of $E$ and the composition $E\twoheadrightarrow B\twoheadrightarrow B''$ imply $\mu_{D,tH}(B'')>\mu_{D,tH}(E)=\mu_{D,tH}(\tilde{E})$, and so $B''$ destabilizes $\tilde{E}$.
\end{proof}
\subsection{The case of the projective plane}
In the case of $\P^2$ and the stability conditions $\sigma_{s,t}$, a wall for a Chern character $v$ is produced when there is an object $E$ with $ch(E)=v$ and an inclusion $A\hookrightarrow E$ in some $\mathcal{A}_{s_0}$ such that
$$
\mu_{s_0,t}(A)=\mu_{s_0,t}(E).
$$ 

Using the explicit formula for $\mu_{s,t}$, it is proven in \cite{ABCH} that the walls are nested semicircles in the $(s,t)$-upper half plane with center on the $s$-axis. Denote by $W_{ch(A),ch(E)}$ the wall corresponding to the inclusion $A\hookrightarrow E$. 

\begin{lemma}\cite[Lemma 6.3]{ABCH}\label{for bounds} Let $E$ be a coherent sheaf on $\P^2$ which is either a torsion sheaf supported in codimension 1, or a torsion-free sheaf (not necessarily Mumford-semistable) satisfying the Bogomolov inequality:
$$
ch_2(E)<\frac{ch_1(E)^2}{2r(E)}
$$
and suppose $A\rightarrow E$ is a map of coherent sheaves which is an inclusion of $\sigma_{s_0,t_0}$-semistable objects of $\mathcal{A}_{s_0}$ of the same slope for some
$$
(s_0,t_0)\in W:=W_{ch(A),ch(E)}.
$$
Then $A\rightarrow E$ is an inclusion of $\sigma_{s,t}$-semi-stable objects of $\mathcal{A}_s$ of the same slope for every point $(s,t)\in W$. 
\end{lemma}

Lemma \ref{for bounds} was used in \cite{ABCH} to provided specific bounds on the radius of the walls and via an identification of $\sigma_{s,t}$-stability with quiver stability, it is shown that if $E$ is a Mumford stable torsion-free sheaf of primitive Chern vector $v$ then there are finitely many isomorphism types of moduli spaces of $\sigma_{s,t}$-stable objects with invariants $v$, i.e., finitely many walls intersecting the slice $\{\sigma_{s,t}\}_{s,t\in \R;t>0}\subseteq \Stab(\P^2)$.

 We finish this section by recalling that for a primitive Chern vector $v$ the moduli space $M_H(v)$ of semistable torsion-free plane sheaves of type $v$ is smooth and a Mori dream space (see \cite{HL}, or \cite{BMW} for a detailed explanation). It is shown in \cite{ABCH} that above the outermost wall $M_{s,t}(v)\cong M_H(v)$, and the argument given in \cite{BMW} shows that decreasing $t$ corresponds to running a directed minimal model program for $M_H(v)$ so that each $M_{s,t}(v)$ is a birational model. 
 
 Things are slightly different when studying the Gieseker moduli of 1-dimensional sheaves although most of the arguments are the same. By the work of Le Potier \cite{LP2} we know that these moduli spaces are irreducible, locally factorial and their Picard group is free abelian of rank 2. Moreover, a specific set of generators is given, namely: the determinant line bundle and the line bundle giving the support map. It is not hard then to prove that the Gieseker moduli of 1-dimensional plane sheaves of fixed invariants is also a Mori dream space (an argument can be found in \cite{W}).

%%%%%%%%%%%%%%%%%%%%%%%%%%%%%%%%%%%%%%%%%%%%%%%%%%

 \section{Duality}\label{duality}
Let $X$ be a smooth projective surface, $D,H\in N^1(X)_{\R}$ with $H$ ample. Let $\sigma_{D,tH}=(Z_{D,tH}, \mathcal{A}_{DH})$ be the stability condition of Theorem \ref{stabilityfunction} and assume that projective coarse moduli spaces for $\sigma_{D,tH}$ and $\sigma_{-D+K,tH}$ are known to exist. For example, $X$ can be $\P^2$ \cite{ABCH}, $\P^1\times \P^1$ or the blow up of $\P^2$ at one point \cite{AM15}, or a $K3$ surface \cite{BM}. This section is devoted to prove the following result: 

\begin{theorem}\label{duality} The functor $\mathcal{F}\mapsto \mathcal{F}^{D}:=R\mathcal{H}om(\mathcal{F},\omega_{X})[1]$ induces an isomorphism between the Bridgeland moduli $M_{D,tH}(ch(F))\cong M_{-D+K_X,tH}(ch(F^D))$ provided that $ch(F)$ is the chern character of an object in $\mathcal{A}_{DH}$ of phase in $(0,1)$. 
\end{theorem}

This result was proven by Maican \cite{MDUAL} for moduli spaces of Gieseker semistable sheaves on $\P^n$ supported on curves. The theorem above recovers Maican's for $X=\P^2$ and $t\gg 0$. For the proof in this context we will need the following

\begin{lemma}\label{lemma dual} Let $E$ be a $\sigma_{D,tH}$-(semi)stable object in $\mathcal{P}_{\sigma_{D,tH}}(0,1)$. Then
\begin{enumerate}
\item If $E$ is $\sigma_{D,tH}$-stable then it is quasi-isomorphic to a two-term complex of vector bundles $E^{-1}\rightarrow E^0$.
\item $\mathcal{H}^{-1}(E)$ is torsion-free with Mumford semistable factors of Mumford slope $<DH$.
\item If $A\in \mathcal{A}_{DH}$ is an object all of whose $\sigma_{D,tH}$-semistable factors belong to $\mathcal{P}_{\sigma_{D,tH}}(0,1)$ then $A^D\in\mathcal{A}_{(-D+K_{X})H}$. 
\item $E^{D}\in\mathcal{A}_{(-D+K_X)H}$ is $\sigma_{-D+K_X,tH}$-(semi)stable.
\item If $E,F\in \mathcal{A}_{DH}$ are $S$-equivalent then so are $E^{D},F^{D}\in\mathcal{A}_{(-D+K_X)H}$.   
\item For any flat family $\mathbf{F}\in D^{b}(S\times X)$ of $\sigma_{D,tH}$-semistable objects in $\mathcal{A}_{DH}$ with fibers of invariants $ch(F_s)$ there is a flat family $\mathbf{F}^{D}\in D^{b}(S\times X)$ with fibers of invariants $ch({F_s}^D)$ such that
$$
Li_s^*(\mathbf{F}^D)\cong (Li_s^*\mathbf{F})^D.
$$
\end{enumerate}
\end{lemma}
\begin{proof}
Part (a) is the content of \cite[Lemma 10.1(a)]{B2}. For (b), notice that $E\in\mathcal{A}_{DH}$ fits into the short exact sequence 
$$
0\rightarrow \mathcal{H}^{-1}(E)[1]\rightarrow E\rightarrow\mathcal{H}^0(E)\rightarrow 0,
$$
and therefore if $\mu_{max}(\mathcal{H}^{-1}(E))=DH$ then $\mathcal{H}^{-1}(E)[1]$ will have a subobject $T$ in $\mathcal{A}_{DH}$ of phase 1, which will $\sigma_{D,tH}$-destabilize $E$. 

Before proving (c), note that for any coherent sheaf $F$ with Mumford semistable factors of Mumford slope $<DH$ (resp. $>DH$) we have
$$
\mathcal{H}^i(F^D)=\begin{cases}\text{torsion free sheaf with }\mu_{min}>-DH+K_XH\\
 (\text{resp. }\mu_{max}<-DH+K_XH)& \text{if}\ i=-1\\
\text{torsion or}\ 0\ & \text{if}\ i=0\\
\text{0-dim torsion sheaf or}\ 0\ &\text{if}\ i=1\\
0\ &\text{otherwise.}\end{cases}
$$
Indeed, if $F$ is locally free the statement is clear since $F^D=\mathcal{H}om(F,\mathcal{O})\otimes \omega_X[1]=F^*\otimes\omega_X[1]$. For torsion sheaves the statement follows from \cite[Proposition 1.1.6]{HL}. If $F$ is torsion-free then $F$ embeds into its double dual and we get an exact sequence
$$
0\rightarrow F\rightarrow F^{**}\rightarrow T\rightarrow 0,
$$
where $T$ is a 0-dimensional torsion sheaf, and the statement follows after taking cohomology on the exact triangle
$$
T^{D}\rightarrow F^{***}\otimes \omega_X[1]\rightarrow F^D\rightarrow T^D[1].
$$
In the general case, $F$ fits into an exact sequence
$$
0\rightarrow \mathcal{T}\rightarrow F\rightarrow \mathcal{F}\rightarrow 0,
$$ 
where $\mathcal{T}$ is its torsion subsheaf, and $\mathcal{F}$ is torsion-free. Thus the statement follows after taking cohomology on the distinguished triangle
$$
\mathcal{F}^{D}\rightarrow F^D\rightarrow \mathcal{T}^D\rightarrow \mathcal{F}^D[1].
$$
Now, assume that $E$ is $\sigma_{D,tH}$-stable and let us prove that $E^D\in \mathcal{A}_{(-D+K_X)H}$. Taking cohomology on the exact triangle
$$
\mathcal{H}^0(E)^D\rightarrow E^D\rightarrow \mathcal{H}^{-1}(E)[1]^D\rightarrow \mathcal{H}^0(E)^D[1]
$$
we get the long exact sequence of sheaves
\begin{eqnarray*}
0\rightarrow\mathcal{H}^{-1}(\mathcal{H}^0(E)^D)\rightarrow\mathcal{H}^{-1}(E^D)\rightarrow\mathcal{H}^{-1}(\mathcal{H}^{-1}(E)[1]^D)\rightarrow \\
\rightarrow \mathcal{H}^0(\mathcal{H}^0(E)^D)\rightarrow\mathcal{H}^0(E^D)\rightarrow\mathcal{H}^0(\mathcal{H}^{-1}(E)[1]^D)\rightarrow\mathcal{H}^1(\mathcal{H}^0(E)^D)\rightarrow 0
\end{eqnarray*}
since by (a) $E^D$ is a two-term complex of vector bundles. But $\mathcal{H}^{-1}(E)[1]^D=\mathcal{H}^{-1}(E)^D[-1]$ and so 
$$
\mathcal{H}^{-1}(\mathcal{H}^{-1}(E)[1]^D)=0.
$$
This implies that $\mathcal{H}^{-1}(E^D)\in \mathcal{F}_{(-D+K_X)H}$, $\mathcal{H}^0(\mathcal{H}^0(E)^D)$ is the torsion subsheaf of $\mathcal{H}^0(E^D)$ and because $\mathcal{H}^1(\mathcal{H}^0(E)^D)$ is zero-dimensional we have $\mathcal{H}^0(E^D)\in\mathcal{Q}_{(-D+K_X)H}$.

Moreover, this proves that  if $A\in \mathcal{P}_{\sigma_{D,tH}}(0,1)$ is $\sigma_{D,tH}$-stable, then $A^D$ is an element of $\mathcal{A}_{(-D+K_X)H}$. For arbitrary $A\in\mathcal{P}_{\sigma_{D,tH}}(0,1)$, $A$ is in the extension closure of some $\sigma_{D,tH}$-stable objects $A_1,\dots,A_k\in\mathcal{A}_{DH}$ of the same phase and so 
$$
A^D\in\langle {A_1}^D,\dots, {A_k}^D\rangle\subset \mathcal{A}_{(-D+K_X)H}.
$$
By the same argument we get (c).
 
Assume for the moment that $E$ is $\sigma_{D,tH}$-stable. Let us prove that there is no injective map
$$
0\rightarrow \mathcal{K}\rightarrow E^D
$$
in $\mathcal{A}_{(-D+K)H}$ with $\mathcal{K}$ having at least one of its Bridgeland semistable factors of phase 1. If so, there would be an inclusion
$$
0\rightarrow A\rightarrow E^D
$$
with $A\in\mathcal{P}_{\sigma_{-D+K,tH}}(1)$ stable, i.e., $A=\C_x$ or $A=F[1]$ for some locally free sheaf $F$ with $\mu_H$-semistable factors of slope $(-D+K_X)H$ (\cite[Lemma 10.1(b)]{B2}). But if $E$ is stable then $E^D$ is derived equivalent to a two-term complex of vector bundles implying $\Hom(\C_x,E^D)=0$. Also 
$$
\Hom(F[1],E^D)=\Hom(F,\mathcal{H}^{-1}(E^D))=0
$$
in virtue of Schur's lemma.

Let us now prove that $E^D$ must be $\sigma_{-D+K_X,tH}$-stable. Indeed, if there is a destabilizing sequence
$$
0\rightarrow A\rightarrow E^D\rightarrow B\rightarrow 0
$$
in $\mathcal{A}_{(-D+K_X)H}$, we can choose $B$ to be $\sigma_{-D+K_X,tH}$-stable and by the argument above we know that all $\sigma_{-D+K_X,tH}$-semistable factors of $A$ have phase in $(0,1)$, then by dualizing this sequence we get a destabilizing sequence for $E$ in $\mathcal{A}_{DH}$ since
$$
\mu_{D,tH}(\cdot)=-\mu_{-D+K_X,tH}(\cdot)^D.
$$

We conclude that $E^D$ is $\sigma_{-D+K_X,tH}$-semistable for all $\sigma_{D,tH}$-semistable objects $E$ of phase in $(0,1)$ just by dualizing the Jordan-H\"older filtration of $E$: Let $0=F_0\subset F_1\subset F_2\subset\cdots \subset F_n=E$ be a $\sigma_{D,tH}$-Jordan-H\"older filtration for $E$ in $\mathcal{A}_{DH}$, then $\left(E/F_n\right)^D\subset \left(E/F_{n-1}\right)^D\subset\cdots\subset E^D$ is a $\sigma_{-D+K_X,tH}$-Jordan-H\"older filtration for $E^D$ in $\mathcal{A}_{(-D+K_X)H}$ with stable factors $(F_i/F_{i-1})^D$. This also gives part (d).

For the last part let $\mathbf{F}^D:=R\mathcal{H}om(\mathbf{F},\omega_{S\times X/S})$ then
$$
Li_s^*(R\mathcal{H}om(\mathbf{F},\omega_{S\times X/S}))\cong R\mathcal{H}om(Li_s^*\mathbf{F},\omega_{X})\in\mathcal{A}_{(-D+K_X)H}.
$$
\end{proof}

\begin{proofof}[Theorem \ref{duality}] Every flat family $\mathbf{F}\in D^{b}(S\times X)$ of topological type $v$ gives a morphism $\pi:S\rightarrow M_{D,tH}(v)$ and a morphism $\pi^D:S\rightarrow M_{-D+K_X,tH}(v^D)$ corresponding to the family $\mathbf{F}^D$ of the lemma. Since $\pi^D$ is constant on the fibers of $\pi$ by part (e) then $\pi^D$ factors through a morphism $M_{D,tH}(v)\rightarrow M_{-D+K_X,tH}(v^D)$ which sends the closed point representing $E$ to the closed point representing $E^D$. The symmetry of the situation and the fact that $(\_)^{DD}=Id$ prove that such morphism is an isomorphism. 
\end{proofof}
\begin{remark} In the special case when $X=\P^2$ and $v=(0,d,-3d/2)$ duality gives an automorphism $(\_)^D:M_{-3/2,t}(v)\cong M_{-3/2,t}(v)$ for all $t>0$.
\end{remark}

 %\begin{corollary}\label{duality2} Let $N([C],\chi)$ denote the moduli space of Gieseker semistable sheaves of Euler-Poincare characteristic $\chi$ supported on a curve of class $[C]$. The functor $\mathcal{F}\mapsto \mathscr{E}xt^1(\mathcal{F},\omega_{X})$ induces an isomorphism between $N([C],\chi)$ and $N([C],-\chi)$.
 %\end{corollary}
 
 \begin{corollary}\label{duality2}
 Let $[C]\in \NS(X)$ be a curve class and $H\in \Amp(X)$. Then the functor $\mathcal{F}\mapsto \mathscr{E}xt^1(\mathcal{F},\omega_{X})$ induces an isomorphism between the moduli spaces of $H$-Gieseker semistable sheaves $M_H(0,[C],ch_2)$ and $M_H(0,[C],K_X\cdot C-ch_2)$.
 \end{corollary}
 \begin{proof} Take $D=K_X/2$ in the duality theorem. Then
 \begin{equation}\label{slopes for corollary 1}
  \mu_{K_X/2,tH}(E)=\frac{ch_2(E)-ch_1(E)\cdot\frac{K_X}{2}+\frac{ch_0(E)}{2}(\frac{K_X^2}{4}-t^2H^2)}{(ch_1(E)-ch_0(E)\frac{K_X}{2})tH}.
 \end{equation}
 Thus, if $ch_0(E)=0$ and $ch_1(E)=[C]$ then
 \begin{equation}\label{slopes for corollary 2}
 \mu_{K_X/2,tH}(E)=\frac{ch_2(E)-C\cdot\frac{K_X}{2}}{tC\cdot H}=\frac{\chi(E)}{tC\cdot H}
 \end{equation}
 and therefore a sheaf of those invariants that is $\sigma_{K_X/2,tH}$-semistable is also $H$-Gieseker semistable. By \cite[Theorem 1.1]{LQ} we know that the values of $t$ for which there is an inclusion of objects $A\hookrightarrow E$ with $\mu_{K_X/2,tH}(A)=\mu_{K_X/2,tH}(E)$ is bounded above (this also follows by a result of Maciocia \cite[Theorem 3.11]{MACIO} when considering the family of stability conditions $\sigma_{K_X/2+sH,tH}$). If $E$ is an object with $ch_0(E)=0$ that is $\sigma_{K_X/2,tH}$-semistable for all $t\gg0$ then $E$ must be a sheaf since otherwise the natural inclusion $\mathcal{H}^{-1}(E)[1]\hookrightarrow E$ would destabilize $E$ for large values of $t$. Indeed, if $\mathcal{E}=\mathcal{H}^{-1}(E)$ is nonzero then
 \begin{align*}
  \mu_{K_X/2,tH}(\mathcal{E}[1])&=\frac{-ch_2(\mathcal{E})+ch_1(\mathcal{E})\cdot\frac{K_X}{2}-ch_0(\mathcal{E})(\frac{K_X^2}{8}-\frac{t^2}{2}H^2)}{-(ch_1(\mathcal{E})-ch_0(\mathcal{E})tH}>\frac{\chi(E)}{tC\cdot H} 
 \end{align*}
 for $t\gg 0$ because $\mathcal{E}$ is torsion-free.
 
 Conversely, let $E$ be a $H$-Gieseker semistable sheaf with $ch(E)=(0,[C],ch_2)$. If $A\hookrightarrow E$ is a subobject in $\mathcal{A}_{K_XH/2}$ that $\sigma_{K_X/2,tH}$-destabilizes $E$ for some $t$ then $A$ must be a sheaf, because $E$ is a sheaf, and of positive rank since otherwise it would destabilize $E$ with respect to $H$-Gieseker stability. Now, since $ch_0(A)>0$ then it follows from equations \eqref{slopes for corollary 1} and \eqref{slopes for corollary 2} that for $t\gg0$ 
 $$
 \mu_{K_X/2,tH}(A)<\mu_{K_X/2,tH}(E)
 $$
 and so the inclusion $A\hookrightarrow E$ must produce a wall. Since the walls are bounded above we conclude that above all walls $E$ is $\sigma_{K_X/2,tH}$-semistable. The coarse moduli spaces $M_H(0,[C],ch_2)$ were constructed by Simpson \cite{Simp} via invariant theory, and the conclusion follows from the duality theorem.
 \end{proof}

%%%%%%%%%%%%%%%%%%%%%%%%%%%%%%%%%%%%%%%%%%%%%%%%%%

 \section{Bridgeland Walls for 1-Dimensional Plane Sheaves}\label{estimates} 
In this section and for the reminder of the paper $X=\P^2$. As mentioned in Section \ref{prelim}, moduli spaces of Gieseker semistable plane sheaves of Hilbert polynomial $P(t)=ct+\chi$ were studied by Le Potier in \cite{LP2}, where it is shown that these moduli spaces are irreducible, locally factorial, and smooth at the stable points. For small values of $c$ it is possible to find nice stratifications of these moduli spaces by studying their resolutions, see \cite{DM} for $c=4$ and  \cite{M5} and \cite{M6} for $c=5$ and $6$. Studying a sheaf by studying its possible resolutions is same as replacing such sheaf for an equivalent element in the derived category. Indeed, each strata in the stratifications given by Dr\'ezet and Maican in \cite{DM} and by Maican in \cite{M5} and \cite{M6} can be interpreted as a set of extensions in a tilted category \cite{BMW}. 
   
Moreover, as in \cite{BMW}, each set of extensions produces a Bridgeland wall, and these are all the walls for the directed MMP. The following numerical bound coming from Lemma \ref{for bounds} produces some of these sets of extensions for arbitrary value of $c$ even when Maican-type stratifications are unknown.

Let $E$ be a sheaf of topological type $(0,c,d)$ with $c>0$ and let $F$ be a destabilizing object (which is necessarily a sheaf), then $E$ and $F$ fit into an exact sequence
$$
0\rightarrow K\rightarrow F\rightarrow E
$$   
of coherent sheaves. By Lemma \ref{for bounds} we must have $K\in \mathcal{F}_s$ and $F\in\mathcal{Q}_s$ for all $s$ along the wall
$$
W_{ch(F),ch(E)}=\left\{(s,t)\colon \mu_{s,t}(F)=\mu_{s,t}(E)\right\}=\left\{(s,t)\colon A\frac{t^2}{2}+A\frac{s^2}{2}+Bs=C\right\},
$$
where 
\begin{align*}
A&=ch_0(F)ch_1(E)-ch_0(E)ch_1(F),\\
B&=ch_2(F)ch_0(E)-ch_2(E)ch_0(F),\\
C&=ch_2(F)ch_1(E)-ch_2(E)ch_1(F).
\end{align*}
If $ch(F)=(r',c',d')$, then in our case this wall is a semicircle with center 
$$
\left(-\frac{B}{A},0\right)=\left(\frac{ch_2(E)ch_0(F)}{ch_0(F)ch_1(E)},0\right)=\left(\frac{d}{c},0\right),
$$ 
and radius 
$$
R=\sqrt{\left(\frac{B}{A}\right)^2+2\frac{C}{A}}=\sqrt{\left(\frac{d}{c}\right)^2+2\frac{d'}{r'}-2\frac{dc'}{r'c}}.
$$ 
Therefore, we have
$$
\frac{ch_1(K)}{ch_0(K)}\leq \frac{d}{c}-R\ \ \ \text{ and }\ \ \ \ \frac{c'}{r'}\geq \frac{d}{c}+R. 
$$
Since $ch_0(K)=r'$ and $ch_1(K)-c'+c\geq 0$, then combining the inequalities above we get
\begin{equation}\label{bound1}
R+\frac{d}{c}\leq \frac{c'}{r'}\leq \frac{d}{c}+\frac{c}{r'}-R,
\end{equation}
which immediately produces
\begin{equation}\label{boundradius}
R\leq \frac{c}{2r'}.
\end{equation}

The purpose of the next section is to study Veronese surfaces and their secant varieties seeing these as special loci on a moduli space of Bridgeland semistable objects. To identify which moduli space we should look at, consider the $(d-3)$-uple embedding of $\P^2$ for an integer $d>3$,
$$
\nu_{d-3}\colon \P^2\rightarrow \P(H^0(\P^2,\mathcal{O}(d-3))^{\vee})\cong \P(\Ext(1,\mathcal{O}(d),\mathcal{O}[1])),
$$
where the last isomorphism is given by Serre duality. Thus, if $\mathcal{O}(d)$ and $\mathcal{O}[1]$ are both objects of a full abelian subcategory $\mathcal{A}\subset D^b(\P^2)$, then it would be possible to see the $(d-3)$-Veronese surface as a locus of non-split extensions 
$$
0\rightarrow \mathcal{O}[1]\rightarrow E\rightarrow \mathcal{O}(d)\rightarrow 0
$$
in $\mathcal{A}$. The topological type of all such extensions is 
$$
ch(E)=ch(\mathcal{O}[1])+ch(\mathcal{O}(d))=(0,d,\frac{d^2}{2}).
$$
When $d$ is odd, we can twist by the line bundle $\mathcal{O}((-d-3)/2)$ and obtain the invariants $v=ch((E(-d-3)/2))=(0,d,-3d/2)$. When $d$ is even, we can twist by $\mathcal{O}((-d-2)/2)$ and obtain the invariants $(0,d,-d)$. 

The Gieseker moduli spaces $M_H(0,d,-3d/2)$ (for $d$ odd) and $M_H(0,d,-d)$ (for $d$ even) behave similarly for our purposes, so we will concentrate in the case when $d$ is odd, and study the case when $d$ is even at the end of Section 6, where the differences will be explained. Nevertheless, the techniques used to study both cases are the same and (except for sections 6.2 and 6.3) every result for $d$ odd has a similar statement for $d$ even. 

Fix the numerical class $v=(0,d,-\frac{3}{2}d)$ with $d$ odd. One of the key ingredients in the computation that follows is the existence of a collapsing wall. The generic element of $M_H(v)$ corresponds to a sheaf $E$ satisfying $H^0(E)=0$ (because $\chi(E)=0$). By using the Beilinson spectral sequence (as in \cite[Section 2.2]{DM} or \cite[Proposition 6.1.1]{M6}) one can conclude that the general element of  $M_H(v)$ has a resolution of the form
$$
0\rightarrow d\mathcal{O}(-2)\rightarrow d\mathcal{O}(-1)\rightarrow E\rightarrow 0.
$$
Since $\mathcal{O}(-2)[1]$ and $\mathcal{O}(-1)$ are objects in $\mathcal{A}_{-3/2}$, then the general element of $M_H(v)$ fits into an exact sequence 
$$
0\rightarrow d\mathcal{O}(-1)\rightarrow E\rightarrow d\mathcal{O}(-2)[1]\rightarrow 0
$$
in $\mathcal{A}_{-3/2}$. In particular $\mathcal{O}(-1)$ produces a wall contracting an open set. By Riemann-Roch the radius of a wall produced by a $\sigma_{s,t}$-destabilizing subobject $F$ can be written as
$$
R=\sqrt{\frac{9}{4}+2\frac{ch_2(F)}{ch_0(F)}+3\frac{ch_1(F)}{ch_0(F)}}=\sqrt{\frac{1}{4}+\frac{2\chi(F)}{ch_0(F)}}.
$$
Thus, the wall $W_{ch(\mathcal{O}(-1)),v}$ has center $(-3/2,0)$ and radius 
$$
R=\sqrt{\frac{1}{4}+\frac{2\chi(\mathcal{O}(-1))}{ch_0(\mathcal{O}(-1))}}=\frac{1}{2}.
$$

The complement of such open set is the theta divisor \cite{LP2} and is the set of semistable sheaves that have at least one section, i.e., those that have $\mathcal{O}$ as a subobject. The corresponding wall $W_{ch(\mathcal{O}),v}$ has radius 
$$
R=\sqrt{\frac{1}{4}+\frac{2\chi(\mathcal{O})}{ch_0(\mathcal{O})}}=\frac{3}{2}.
$$ 
Crossing $W_{ch(\mathcal{O}),v}$  corresponds to a divisorial contraction and since $M_H(v)$ has Picard number 2 then there are no walls between $W_{ch(\mathcal{O}(-1)),v}$ and $W_{ch(\mathcal{O}),v}$. This gives a lower bound for the radius of walls corresponding to flips:

\begin{proposition}\label{radiusofflippingwalls}
Let $d>3$ be an odd integer. Let $F$ be a coherent sheaf of positive rank, and let $E$ be a coherent sheaf with $ch(E)=(0,d,-\frac{3}{2}d)$. A morphism of sheaves $F\rightarrow E$ which is an inclusion of objects in the category $\mathcal{A}_{-3/2}$ produces a wall corresponding to a flip if and only if
\begin{equation}\label{higher rank}
\frac{3}{2}<\sqrt{\frac{1}{4}+\frac{2\chi(F)}{ch_0(F)}}\leq \frac{d}{2ch_0(F)}.
\end{equation}
\end{proposition}
\begin{proof} Since the center of all walls of type $v=(0,d,-3d/2)$ in the $(s,t)$-plane is $(-3/2,0)$, then the 1-parameter family of stability conditions $\{\sigma_{-3/2,t}\}_{t>0}$ intersects every chamber. Because of the correspondence between Bridgeland wall-crossing for the topological type $v$ and the MMP for the moduli spaces $M_H(v)$ \cite[Theorem 1.1]{BMW}, we know that starting with $t\gg 0$ and then decreasing $t$ corresponds to run a directed Minimal Model Program on $M_H(v)$. Moreover, because $M_H(v)$ is a Mori dream space of Picard number 2, then the moduli spaces of $\sigma_{-3/2,t}$-semistable objects of topological type $v$ account for all the birational models of $M_H(v)$. Because there are no walls between $W_{ch(\mathcal{O},v)}$  and  $W_{ch(\mathcal{O}(-1),v)}$, then every wall corresponding to a flip has radius larger than the radius of $W_{ch(\mathcal{O},v)}$, i.e., $R>3/2$. The other inequality is \eqref{boundradius}.
\end{proof}
It is useful to know whether the new objects we get after crossing a wall are Bridgeland stable or pseudo-stable, and we can answer this in a very special case:
\begin{proposition}\label{pseudostableisstable} Let $v=(0,d,-3d/2)$ and assume that $E\in\mathcal{A}_{-3/2}$ is an object with $ch(E)=v$ that is Bridgeland semistable with a Jordan-H\"older filtration of length 1 for a stability condition at a wall $W$ of type $v$. Then $E$ is Bridgeland stable for stability conditions on one of the chambers determined by $W$. 
\end{proposition}
\begin{proof} Assume that the Jordan-H\"older filtration of $E$ at $W$ is
$$
0\rightarrow A\rightarrow E\rightarrow B\rightarrow 0
$$
and that $\mu_{-3/2,t}(A)>\mu_{-3/2,t}(B)$ above $W$, then $E$ is Bridgeland pseudo-stable below $W$ by Proposition \ref{MMP1}. Assume that $E$ is not Bridgeland stable below $W$, then there should be a subobject $E'\hookrightarrow E$ with $ch_0(E')=ch_0(E)=0$ such that 
$$
\mu_{-3/2,t}(E')=\mu_{-3/2,t}(E)\ \ \Leftrightarrow\ \ \frac{\chi(E')}{ch_1(E')}=\frac{\chi(E)}{ch_1(E)}=0.
$$
Thus, $\chi(E')=\chi(E)=0$. Moreover, $ch_1(E')<ch_1(E)=d$ since otherwise the quotient $E/E'$ in $\mathcal{A}_{-3/2}$ would be mapped to 0 by $Z_{-3/2,t}$. Let $K=\ker(E'\twoheadrightarrow B)$ and $ch_1(E')=d-h$, then we have a diagram
$$
\begin{diagram}
\node{K}\arrow{s,J}\arrow{e,J}\node{E'}\arrow{s,J}\arrow{se,A}\node{}\\
\node{A}\arrow{e,J}\node{E}\arrow{e,A}\node{B}
\end{diagram}
$$

If $ch_0(B)=r',\ ch_1(B)=c'$, and $ch_2(B)=\delta'$, then 
\begin{align*}
ch(K)&=(ch_0(K),ch_1(K),ch_2(K))=(-r',d-h-c',-\frac{3}{2}(d-h)-\delta'),\\
ch(A)&=(ch_0(A),ch_1(A),ch_2(A))=(-r',d-c',-\frac{3}{2}d-\delta').
\end{align*}

Note that in this case $\mathfrak{R}(Z_{-3/2,t}(K))=\mathfrak{R}(Z_{-3/2,t}(A))$ and $\mathfrak{I}(Z_{-3/2,t}(K))=\mathfrak{I}(Z_{-3/2,t}(A))-ht$. But $A$ is Bridgeland stable at $W$ and so it is Bridgeland stable for $t$ sufficiently near $W$, therefore
$$
\mu_{-3/2,t}(K)<\mu_{-3/2,t}(A)
$$
for $t$ above and below $W$. This implies that $-h\mathfrak{R}(Z_{-3/2,t}(A))<0$ above and below $W$ and so $h=0$. Thus $E$ is Bridgeland stable.
\end{proof}
\subsection{Rank one walls}
Setting $ch_0(F)=1$ in (\ref{higher rank}) one finds the set of admissible values for the Euler characteristic of a destabilizing object producing a wall corresponding to a flip:
$$
\chi(F)=2,\dots,\frac{d^2-1}{8}.
$$ 

The possible values for the first Chern class come from solving the inequality (\ref{bound1}). Assume $\chi(F)=\frac{d^2-1}{8}-\ell$, for $0\leq \ell\leq \frac{d^2-1}{8}-2$, then 
\begin{equation}\label{inequalityrank1}
\sqrt{\frac{d^2}{4}-2\ell}-\frac{3}{2}\leq ch_1(F)\leq -\frac{3}{2}+d-\sqrt{\frac{d^2}{4}-2\ell}.
\end{equation}

It is easy to check that $ ch_1(F)=(d-3)/2$ is always a solution. These are the invariants of a twisted ideal sheaf $\mathcal{I}_Z((d-3)/2)$ of a 0-dimensional subscheme $Z$ of length $\ell$. 

Now, for a generic 0-dimensional subschemes $Z$ of length $\ell$, Gaeta's theorem states that if we write
$$
\ell=\frac{r(r+1)}{2}+s,\ \ 0\leq s\leq r,
$$ 
then $\mathcal{I}_{Z}$ has a free resolution of the form
$$
0\rightarrow \mathcal{O}(-r-1)^{\oplus(r-2s)}\oplus \mathcal{O}(-r-2)^{\oplus s}\rightarrow \mathcal{O}(-r)^{\oplus(r-s+1)}\rightarrow \mathcal{I}_Z\rightarrow 0
$$
if $r\geq 2s$, or
$$
0\rightarrow  \mathcal{O}(-r-2)^{\oplus s}\rightarrow \mathcal{O}(-r)^{\oplus(r-s+1)}\oplus \mathcal{O}(-r-1)^{\oplus (2s-r)}\rightarrow \mathcal{I}_Z\rightarrow 0
$$
if $r\leq 2s$. 

For $ 0\leq \ell\leq \frac{d^2-1}{8}-2$, a simple computation shows that $d-2r>0$, and $d-2r-2>0$ when $r\leq 2s$. Therefore, for generic 0-dimensional subschemes $Z,W\subset \P^2$ of length $\ell$ we have $\Hom(\mathcal{O},\mathcal{I}_W\otimes^{L}\mathcal{I}_Z(d))\neq 0$. Thus,
\begin{align*}
\Ext(1,\mathcal{I}_W^{\vee}\left((-d-3)/2\right)[1],\mathcal{I}_Z\left((d-3)/2\right))&=\Hom(\mathcal{I}_W^{\vee}\left((-d-3)/2\right),\mathcal{I}_Z\left((d-3)/2\right))\\
&=\Hom(\mathcal{O},\mathcal{I}_W\left((d+3)/2\right)\otimes^{L}\mathcal{I}_Z\left((d-3)/2\right))\\
&=\Hom(\mathcal{O},\mathcal{I}_W\otimes^{L}\mathcal{I}_Z(d))\neq 0.
\end{align*}
Since $\mathcal{I}_W^{\vee}\left((-d-3)/2\right)[1]\in\mathcal{A}_{-3/2}$ by Theorem \ref{duality} with $t\gg 0$, then there are nontrivial extensions
$$
0\rightarrow \mathcal{I}_Z\left((d-3)/2\right)\rightarrow E\rightarrow \mathcal{I}_W^{\vee}\left((-d-3)/2\right)[1]\rightarrow 0
$$
producing sheaves  $E\in \mathcal{A}_{-3/2}$. The generic extension is Bridgeland stable above the wall determined by  $\mathcal{I}_Z\left((d-3)/2\right)$, i.e., for 
$$
t>\sqrt{\frac{d^2}{4}-2\ell}.
$$ 
This proves that the number of actual rank 1 walls corresponding to flips is $\displaystyle \frac{d^2-9}{8}$. 

Notice that the exceptional loci for a rank 1 flip is not irreducible in general. Indeed, the inequality (\ref{inequalityrank1}) has unique solution only when $\displaystyle \ell<\frac{d-1}{2}$. However, setting
$$
\mathcal{G}^{W}_{\ell,i}:=\mathcal{I}_{W}\left(\frac{d-3}{2}+i\right),\ \ \ \ \ \text{length}(W)=\ell+\frac{i(d+i)}{2},\ \ \ i\in\Z
$$
we have
\begin{proposition}\label{prop for remark}
If $c=\frac{d-3}{2}+i$ is solution for (\ref{inequalityrank1}) then so is $c=\frac{d-3}{2}-i$. Generically, the corresponding destabilizing objects are of the form $\mathcal{G}^{W}_{\ell,i}$ and $\mathcal{G}^{Y}_{\ell,-i}$ respectively. Moreover, if $E_{\ell,k}$ denotes the component containing the locus of sheaves destabilized by an object of the form $\mathcal{G}^{W}_{\ell,k}$ then $E_{\ell,-k}$ is the image of $E_{\ell,k}$ by the duality automorphism. 
\end{proposition}
\begin{proof}
The first part is a trivial computation. For the second one note that a generic destabilizing sequence is of the form
$$
0\rightarrow \mathcal{G}^{W}_{\ell,i}\rightarrow E\rightarrow (\mathcal{G}^{Y}_{\ell,-i})^{D}\rightarrow 0
$$
which is again a trivial computation of the invariants.
\end{proof}

\begin{remark}\label{hilbertwalls}
It follows from \cite[Sections 9 \& 10]{ABCH} that the outermost wall for the Chern character $w=(1,(d-3)/2,(d-3)^2/8-\ell)$ is produced by the inclusions 
$$
\mathcal{O}((d-5)/2)\hookrightarrow \mathcal{I}_Z((d-3)/2),
$$
destabilizing twisted ideal sheaves of 0-dimensional subschemes $Z\subset \mathbb{P}^2$ of length $\ell$ supported on a line. This wall is a semicircle with center $\left(\frac{d-5}{2}-\ell,0\right)$ and radius $\ell-\frac{1}{2}$. In the outermost chamber, the only Bridgeland semistable (actually stable) objects of Chern character $w$ are the twisted ideal sheaves $\mathcal{I}_Z((d-3)/2)$, and the corresponding Bridgeland moduli space is isomorphic to $\Hilb^{\ell}(\mathbb{P}^2)$. 
\end{remark} 
\begin{remark}\label{hilbertwalls2}
With the notation of Remark \ref{hilbertwalls}. Notice that the right intercept of a wall $W$ of radius $R$ for the Chern character $v$ and the $s$-axis is $-\frac{3}{2}+R$, and the right intercept of the outermost wall for the  Chern character $w$ and the $s$-axis is $\frac{d-5}{2}$. Therefore, it follows from Remark \ref{hilbertwalls}  that as long as $R>\frac{d}{2}-1$, the only Bridgeland semistable objects of Chern character $w$ along $W$ are the twisted ideal sheaves $\mathcal{I}_Z((d-3)/2)$, where $Z\subset\mathbb{P}^2$ is a 0-dimensional subscheme of length $\ell$.
\end{remark}

\begin{remark}\label{hilbertwalls3} 
Let $W$ be a wall of radius $R$ for the Chern character $v$ produced by a destabilizing subobject $A\rightarrow E$ (here $E$ is a sheaf with $ch(E)=v$). Because $R\leq \frac{d}{2ch_0(A)}$ and $d\geq 5$, then each wall of radius $R>\frac{d-2}{2}$ must be a rank 1 wall. Moreover, since inequality \eqref{inequalityrank1} has only one solution when 
$$
R>\frac{d-2}{2},
$$
then the walls of radius $R>\frac{d-2}{2}$ are produced by subobjects $A$ with Chern character 
$$
ch(A)=\left(1,\frac{d-3}{2},\frac{(d-3)^2}{8}-\ell\right),\ \ 0\leq \ell<\frac{d-1}{2}.
$$
Since $A$ is Bridgeland semistable along $W$, then by Remark \ref{hilbertwalls2} we must have $A=\mathcal{I}_Z((d-3)/2)$ for some 0-dimensional subscheme $Z\subset \mathbb{P}^2$ of length $\ell$.
\end{remark}

%%%%%%%%%%%%%%%%%%%%%%%%%%%%%%%%%%%%%%%%%%%%%%%%%%

\section{The Embedded Problem: Flips of Secant Varieties.} 

In \cite{V1} and \cite{V2}, Vermeire describes a sequence of flips for the secant varieties of an embedding $X\hookrightarrow \P^N$ of an algebraic surface. This sequence of flips is constructed in similar fashion to the flips obtained by Thaddeus \cite{THA} when studying variation of GIT for moduli spaces of stable pairs on curves. The first of these flips is easy to describe and it is the content of \cite[Theorem 4.12]{V1}. Roughly speaking, if the embedding of $X$ is sufficiently ample such that it can be generated by quadrics with only linear syzygies then there is a flip diagram
$$
\begin{diagram}
\node{}\node{\tilde{M}}\arrow{sw,t}{\pi}\arrow{se,t}{h}\node{}\\
\node{\bl_X(\P^N)}\arrow{s}\arrow{se,t}{\varphi^+}\node{}\node{M}\arrow{sw,b}{\varphi^-}\\
\node{\P^N}\arrow{e,b,..}{\varphi}\node{\P^s}\node{}
\end{diagram}
$$ 
where $\varphi :\P^N\dashrightarrow \P^s$ is the rational map given by the forms defining $X$ and $\tilde{M}$ is the blow-up of $\bl_X(\P^N)$ along the strict transform of the secant variety $\widetilde{Sec X}$. The diagram restricts to 
$$
\begin{diagram}
\node{}\node{E}\arrow{sw,t}{\pi}\arrow{se,t}{h}\node{}\\
\node{\P(\mathcal{E})}\arrow{se,b}{\varphi^+}\node{}\node{\P(\mathcal{F})}\arrow{sw,b}{\varphi^-}\\
\node{}\node{\Hilb^2(X)}\node{}
\end{diagram}
$$ 
where $\P(\mathcal{E})\cong \widetilde{Sec X}$ and $\mathcal{F}= {\varphi^+}_*(N^*_{\P(\mathcal{E})/\bl_X(\P^N)}\otimes \mathcal{O}_{\P(\mathcal{E})}(-1))$.

We will see that in the case of the $(d-3)$-uple embedding of $\P^2$, $X=\nu_{d-3}(\P^2)$, such flips appear naturally when running the MMP for the Gieseker moduli $M_H(0,d,-3d/2)$ for $d$ odd (or $M_H(0,d,0)$ for $d$ even).

From Remark \ref{hilbertwalls3} we know that the first $(d-1)/2$ walls of type $v=(0,d,-3d/2)$ ($d$ odd) correspond to flips and are produced by destabilizing subobjects of the form 
$$
\mathcal{G}^{Z}_{\ell,0}=\mathcal{I}_Z((d-3)/2),\ \ \text{length}(Z)=\ell,\ \ 0\leq \ell<\frac{d-1}{2}.
$$
The radius of the wall $W_{ch(\mathcal{G}^{Z}_{\ell,0}),v}$ is 
$$
R_{\ell}=\sqrt{\frac{d^2}{4}-2\ell}.
$$
We denote by $M_{\ell}$ the moduli space of $\sigma_{-3/2,R_{\ell}}$-semistable objects of topological type $v$, and by $M_{\ell}^{\pm}$ the moduli spaces of  $\sigma_{-3/2,R_{\ell}\pm\epsilon}$-semistable objects of topological type $v$ for $0<\epsilon\ll 1$. Thus, for instance, $M_{\ell}^-=M_{\ell+1}^+$. Also, denote by $E_{\ell}^{\pm}$ the exceptional loci for the contractions 
$$
\pi_{\ell}^{\pm}\colon M_{\ell}^{\pm}\rightarrow M_{\ell}.
$$
%Recall that if $A\hookrightarrow E$ is a Bridgeland destabilizing subobject of a Gieseker semistable sheaf $E$ of invariants $ch(E)=(0,d,-\frac{3}{2}d)$ in the category $\mathcal{A}_{-3/2}$, then the corresponding wall has radius
%$$
%R=\sqrt{\frac{1}{4}+\frac{2\chi(A)}{ch_0(A)}}\leq \frac{d}{2ch_0(A)}.
%$$
%This upper bound is realized by $\mathcal{O}((d-3)/2)$ and corresponds to a flip by Proposition \ref{radiusofflippingwalls}. 

Thus we obtain the exceptional loci for the first flip of $M_H(0,d,-3d/2)$ ($d\geq 5$ odd):
\begin{eqnarray*}
E_0^+ &:& 0\rightarrow \mathcal{O}\left((d-3)/2\right)\rightarrow F\rightarrow \mathcal{O}\left((-d-3)/2\right)[1]\rightarrow 0 \\
E_0^- &:& 0\rightarrow \mathcal{O}\left((-d-3)/2\right)[1] \rightarrow G^{\bullet}\rightarrow \mathcal{O}\left((d-3)/2\right) \rightarrow0,
\end{eqnarray*}
these are obtained from the set-theoretic wall-crossing since the objects $\mathcal{O}\left((d-3)/2\right)$ and $\mathcal{O}\left((-d-3)/2\right)[1]$ are $\sigma_{-3/2,t}$-stable for every value of $t$, which follows from \cite[Proposition 6.2]{ABCH} and Theorem \ref{duality}. $E_0^+$ and $E_0^-$ are projective spaces, indeed: 
\begin{eqnarray*}
E_0^+&\cong& \P\left(\Ext(1,\mathcal{O}\left((-d-3)/2\right)[1], \mathcal{O}\left((d-3)/2\right))\right),\\
E_0^-&\cong& \P\left(\Ext(1,\mathcal{O}\left((d-3)/2\right),\mathcal{O}\left((-d-3)/2\right)[1])\right)\\
&=& \P\left(\Ext(2,\mathcal{O}\left((d-3)/2\right),\mathcal{O}\left((-d-3)/2\right))\right)\\
&=& \P(H^2(\P^2,\mathcal{O}(-d)))\\
&\cong& \P(H^0(\P^2,\mathcal{O}(d-3))^{\vee}).
\end{eqnarray*}
There is a natural $\P^2$ embedded in $E_0^-$ by the complete linear series $\nu_{d-3}\colon \P^2\rightarrow E_0^-$. The Veronese surface $X:=\nu_{d-3}(\P^2)$ can be described in terms of extensions, it is the set of complexes $G^{\bullet}$ fitting into a commutative diagram 
\begin{equation}\label{diagram for X}
\begin{diagram}\dgARROWLENGTH=1.7em
\node{}\node{}\node{\mathcal{I}_p\left((d-3)/2\right)}\arrow{sw}\arrow{s,J}\\
\node{\mathcal{O}\left((-d-3)/2\right)[1]}\arrow{e,J}\arrow{s,=}\node{G^{\bullet}_p}\arrow{e,A}\arrow{s}\node{\mathcal{O}\left((d-3)/2\right)}\arrow{s,A}\\
\node{\mathcal{O}\left((-d-3)/2\right)[1]}\arrow{e,J}\node{\mathcal{G}}\arrow{e,A}\node{\C_p}\\
\end{diagram}
\end{equation}
Note that $\mathcal{G}$ is unique (up to scalars) since $\ext(1,\C_p,\mathcal{O}\left((-d-3)/2\right)[1])=1$. Thus $G^{\bullet}_p$ is the image under the pullback homomorphism 
$$
\Ext(1,\C_p,\mathcal{O}\left((-d-3)/2\right)[1])\hookrightarrow \Ext(1,\mathcal{O}\left((d-3)/2\right),\mathcal{O}\left((-d-3)/2\right)[1]).
$$
But we know that 
\begin{align*}
\Ext(1,\C_p,\mathcal{O}\left((-d-3)/2\right)[1]) &\cong \Ext(1,(\mathcal{O}\left((-d-3)/2\right)[1])^D,(\C_p)^D)\\
&=\Ext(1,\mathcal{O}\left((d-3)/2\right),(\C_p)^D),
\end{align*}
so $G^{\bullet}_p$ is also the image under the push-forward map
$$
\Ext(1,\mathcal{O}\left((d-3)/2\right),(\C_p)^D)\hookrightarrow \Ext(1,\mathcal{O}\left((d-3)/2\right),\mathcal{O}\left((-d-3)/2\right)[1]).
$$

Applying the functor $(\_)^D$ to the pullback diagram above gives us the push-forward diagram
$$
\begin{diagram}\dgARROWLENGTH=1.7em
\node{\C_p^D}\arrow{e,J}\arrow{s,J}\node{\mathcal{G}^D}\arrow{s}\arrow{e,A}\node{\mathcal{O}\left((d-3)/2\right)}\arrow{s,=}\\
\node{\mathcal{O}\left((-d-3)/2\right)[1]}\arrow{e,J}\arrow{s,A}\node{(G^{\bullet}_p)^D}\arrow{e,A}\arrow{sw}\node{\mathcal{O}\left((d-3)/2\right)}\\
\node{\left(\mathcal{I}_p((d-3)/2)\right)^D}
\end{diagram}
$$
\begin{proposition}
The elements of $E_0^-$ are fixed by the duality automorphism.
\end{proposition}
\begin{proof} From the discussion above we know that $G^{\bullet}_p=(G^{\bullet}_p)^D$, and so the duality automorphism which restricts to an automorphism $(\_)^D|_{E_0^-}:E_0^-\rightarrow E_0^-$ fixes $X$. Since every automorphism of $E_0^-\cong\P^N$ is linear, then $(\_)^D|_{E_0^-}$ is the identity. 
\end{proof}

The exceptional loci for the second flip are
\begin{eqnarray*}
E_1^+ &:& 0\rightarrow \mathcal{I}_p((d-3)/2)\rightarrow F\rightarrow\mathcal{I}_q^{\vee}((-d-3)/2)[1]\rightarrow 0;\ \ \ p,q\in\P^2 \\
E_1^- &:& 0\rightarrow \mathcal{I}_q^{\vee}((-d-3)/2)[1]\rightarrow G^{\bullet}\rightarrow  \mathcal{I}_p((d-3)/2)\rightarrow 0;\ \ \ p,q\in\P^2. 
\end{eqnarray*} 

This description of $E_1^-$ is given by Proposition \ref{MMP3} since by Remark \ref{hilbertwalls3} and Theorem \ref{duality} we know that both $\mathcal{I}_p((d-3)/2)$ and $\mathcal{I}_q^{\vee}((-d-3)/2)[1]$ are Bridgeland stable along the wall $W_{ch(\mathcal{I}_p((d-3)/2)),v}$.

The following vanishing theorem will be used repeatedly for the rest of the section:
\begin{theorem}[Bertram-Ein-Lazarsfeld, \cite{BEL}] Assume that $X\subset \P^r$ is (scheme-theoretically) cut out by hypersurfaces of degrees $d_1\geq d_2\geq\cdots\geq d_m$. Then 
$$
H^i(\P^r,\mathcal{I}_X^a(k))=0\ \ \ \text{for all}\ \ i\geq 1
$$  
provided that $k\geq ad_1+d_2+\cdots+d_e-r$, where $e=codim(X,\P^r)$.
\end{theorem}
\begin{lemma}\label{common blowup} The fiber product $M_1^+\times_{M_1}M_1^-$ is isomorphic to the common blow-up $\bl_{E_1^+}M_1^+= \bl_{E_1^-}M_1^-$.
\end{lemma}
\begin{proof}
A proof of this statement was already given in \cite[Section 6.2.1]{BMW} for the case $d=5$, and it generalizes for all $d$ (odd) without change. One notices the following vanishing
\begin{eqnarray*}
\Hom(\mathcal{I}_p((d-3)/2),\mathcal{I}_q^{\vee}((-d-3)/2)[1])&=&0\\
\Ext(2,\mathcal{I}_q^{\vee}((-d-3)/2)[1],\mathcal{I}_q^{\vee}((-d-3)/2)[1])&=&0\\
\Ext(2,\mathcal{I}_p((d-3)/2),\mathcal{I}_p((d-3)/2))&=&0\\
\Ext(2,F,\mathcal{I}_p((d-3)/2))&=&0\\
\Ext(2,\mathcal{I}_q^{\vee}((-d-3)/2)[1],F)&=&0
\end{eqnarray*} 
for every $p,q\in \P^2$ and $F\in E_1^+$. The first is obvious when $p\neq q$, for $p=q$ one uses Serre duality and Bertram-Ein-Lazarsfeld vanishing. The last two are obtained by using Serre duality and the fact that $F$ is Bridgeland stable, which is a consequence of Proposition \ref{pseudostableisstable}. This allows us to get diagrams
$$
\footnotesize{
\begin{diagram}\dgARROWLENGTH=0.8em
\node{}\node{}\node{0}\arrow{s}\node{}\\
\node{}\node{}\node{\Ext(1,\mathcal{I}_q^{\vee}((-d-3)/2)[1],\mathcal{I}_q^{\vee}((-d-3)/2)[1])}\arrow{s}\node{}\\
\node{\hspace{1cm}\Ext(1,F,\mathcal{I}_p((d-3)/2))}\arrow{e}\node{\Ext(1,F,F)}\arrow{se,b}{f}\arrow{e,A}\node{\Ext(1,F,\mathcal{I}_q^{\vee}((-d-3)/2)[1])}\arrow{s}\node{}\\
\node{}\node{}\node{\Ext(1,\mathcal{I}_p((d-3)/2),\mathcal{I}_q^{\vee}((-d-3)/2)[1])}\arrow{s}\node{}\\
\node{}\node{}\node{0}\node{}
\end{diagram}}
$$
and 
$$
\footnotesize{
\begin{diagram}\dgARROWLENGTH=0.8em
\node{}\node{}\node{0}\arrow{s}\\
\node{}\node{}\node{\Ext(1,\mathcal{I}_p((d-3)/2),\mathcal{I}_p((d-3)/2))}\arrow{s}\\
\node{\hspace{1cm}\Ext(1,\mathcal{I}_q^{\vee}((-d-3)/2)[1],F)}\arrow{e}\node{\Ext(1,F,F)}\arrow{se,b}{f}\arrow{e,A}\node{\Ext(1,\mathcal{I}_p((d-3)/2),F)}\arrow{s}\\
\node{}\node{}\node{\Ext(1,\mathcal{I}_p((d-3)/2),\mathcal{I}_q^{\vee}((-d-3)/2)[1])}\arrow{s}\\
\node{}\node{}\node{0}
\end{diagram}}
$$
Then we get an exact sequence
$$
0\rightarrow \ker f\rightarrow \Ext(1,F,F)\rightarrow\Ext(1,\mathcal{I}_p((d-3)/2),\mathcal{I}_q^{\vee}((-d-3)/2)[1])\rightarrow 0,
$$
and two surjections 
\begin{align*}
& \ker f \twoheadrightarrow \Ext(1,\mathcal{I}_q^{\vee}((-d-3)/2)[1],\mathcal{I}_q^{\vee}((-d-3)/2)[1]),\\
& \ker f \twoheadrightarrow \Ext(1,\mathcal{I}_p((d-3)/2),\mathcal{I}_p((d-3)/2)).
\end{align*}
Since the compositions 
$$
\Ext(1,\mathcal{I}_q^{\vee}((-d-3)/2)[1],\mathcal{I}_p((d-3)/2))\rightarrow \Ext(1,F,\mathcal{I}_p((d-3)/2))\rightarrow \Ext(1,F,F),
$$
and 
$$
\Ext(1,\mathcal{I}_q^{\vee}((-d-3)/2)[1],\mathcal{I}_p((d-3)/2))\rightarrow \Ext(1,\mathcal{I}_q^{\vee}((-d-3)/2)[1],F)\rightarrow \Ext(1,F,F),
$$
coincide, then $\ker f$ fits into an exact sequence
$$
\begin{diagram}\dgARROWLENGTH=1.0em
\node{0}\arrow{s}\\
\node{\C}\arrow{s}\\
\node{\Ext(1,\mathcal{I}_q^{\vee}((-d-3)/2)[1],\mathcal{I}_p((d-3)/2))}\arrow{s}\\
\node{\ker f}\arrow{s}\\
\node{\Ext(1,\mathcal{I}_q^{\vee}((-d-3)/2)[1],\mathcal{I}_q^{\vee}((-d-3)/2)[1])\oplus \Ext(1,\mathcal{I}_p((d-3)/2),\mathcal{I}_p((d-3)/2))}\arrow{s}\\
\node{0}
\end{diagram}
$$
Where we take 
\begin{align*}
\C&\cong \Hom(\mathcal{I}_p((d-3)/2),\mathcal{I}_p((d-3)/2)), \ \text{or} \\
\C&\cong\Hom(\mathcal{I}_q^{\vee}((-d-3)/2)[1],\mathcal{I}_q^{\vee}((-d-3)/2)[1]).
\end{align*}
Thus $\ker f$ can be identified with the tangent space of $E_1^+$ at the point $[F]$ and we get an exact sequence
$$
0\rightarrow (\T E_1^+)_{[F]}\rightarrow (\T M_1^+|_{E_1^+})_{[F]}\rightarrow \Ext(1,\mathcal{I}_p((d-3)/2),\mathcal{I}_q^{\vee}((-d-3)/2)[1])\rightarrow 0.
$$
and therefore an exact sequence of sheaves
$$
0\rightarrow \T E_1^+\rightarrow \T M_1^+|_{E_1^+}\rightarrow (\pi_1^+|_{E_1^+})^*((\pi_1^-|_{E_1^-})_*\mathcal{O}_{E_1^-}(1))\rightarrow 0.
$$
Similarly one gets
$$
0\rightarrow \T E_1^-\rightarrow \T M_1^-|_{E_1^-}\rightarrow (\pi_1^-|_{E_1^-})^*((\pi_1^+|_{E_1^+})_*\mathcal{O}_{E_1^+}(1))\rightarrow 0.
$$
This proves that we have a fiber square
$$
\begin{diagram}
\node{\P(N_{E_1^+/M_1^+})\cong \P(N_{E_1^-/M_1^-})}\arrow{s}\arrow{e}\node{E_1^-}\arrow{s,r}{\pi_1^-|_{E_1^-}}\\
\node{E_1^+}\arrow{e,b}{\pi_1^+|_{E_1^+}}\node{\P^2\times\P^2}
\end{diagram}
$$
which completes the proof.
\end{proof}

\begin{proposition}\begin{enumerate}
\item $E_1^+$ and $E_1^-$ are both projective bundles over $\P^2\times \P^2$.
\item $E_1^+\cap E_0^-=X$.
\item The closure of $E_0^-\setminus X$ in $M_1^-$ is isomorphic to the blow-up of $E_0^-$ along $X$.
\end{enumerate}
\end{proposition}
\begin{proof} For part (a), we only need to verify that $\ext(1,\mathcal{I}_q^{\vee}((-d-3)/2)[1],\mathcal{I}_p((d-3)/2))$ and $\ext(1,\mathcal{I}_p((d-3)/2),\mathcal{I}_q^{\vee}((-d-3)/2)[1])$ are constant as $p,q$ vary, because the rest of the argument follows as in \cite[Proposition 4.2]{AB}. We have
\begin{eqnarray*}
\Ext(1,\mathcal{I}_q^{\vee}((-d-3)/2)[1],\mathcal{I}_p((d-3)/2))&=&\Hom(\mathcal{O},\mathcal{I}_p\otimes \mathcal{I}_q(d)),\\
\Ext(1,\mathcal{I}_p((d-3)/2),\mathcal{I}_q^{\vee}((-d-3)/2)[1])&=& \Ext(2,\mathcal{I}_p((d-3)/2),\mathcal{I}_q^{\vee}((-d-3)/2))\\
&\cong& \Hom(\mathcal{O},\mathcal{I}_p\otimes \mathcal{I}_q(d-3))^{\vee}.
\end{eqnarray*}
Note that we can use ordinary tensor instead of derived tensor, because ideal sheaves have a two-term resolution by locally free sheaves. For $p\neq q$ this follows from standard calculations. For $p=q$ one gets constant dimension because
$$
H^1(\P^2, \mathcal{I}_p^2(k))=0\ \  \text{for}\ \  k>0
$$
which follows for example by Bertram-Ein-Lazarsfeld vanishing.

For part (b), diagram (\ref{diagram for X}) already shows that $X\subset E_1^+\cap E_0^-$ since every $G^{\bullet}_p$ admits an injective map (in $\mathcal{A}_{-3/2}$) from $\mathcal{I}_p((d-3)/2)$. For the other inclusion, note that $\hom(\mathcal{I}_p((d-3)/2),\mathcal{O}((d-3)/2))=1$ and $\hom(\mathcal{I}_p((d-3)/2),\mathcal{O}((-d-3)/2)[1])=0$. Therefore, if $E$ fits into a diagram
$$
\begin{diagram}\dgARROWLENGTH=1.7em
\node{}\node{\mathcal{I}_p((d-3)/2)}\arrow{s,J}\node{}\\
\node{\mathcal{O}\left((-d-3)/2\right)[1]}\arrow{e,J}\node{E}\arrow{e,A}\node{\mathcal{O}\left((d-3)/2\right),}
\end{diagram}
$$
then the composition $\mathcal{I}_p((d-3)/2)\hookrightarrow E\twoheadrightarrow \mathcal{O}((d-3)/2)$ is the natural inclusion and so $E\cong G^{\bullet}_p$.

More can be said: since $E_0^-$ is fixed by the duality automorphism, then $E_1^+$ intersects $E_0^-$ along a section over the diagonal $\Delta\subset \P^2\times\P^2$. Since the morphism $\pi_1^+:M_1^+\rightarrow M_1$ collapses the fibers of $E_1^+$, then $\pi_1|_{E_0^-}:E_0^-\rightarrow M_1$ is a closed immersion. By Lemma \ref{common blowup} we have a diagram
$$
\begin{diagram}\dgARROWLENGTH=1.5em
\node{}\node{\bl_{E_1^+}M_1^+}\arrow{s}\arrow{se}\node{}\\
\node{\bl_{X}E_0^-}\arrow{ne}\arrow{s}\node{M_1^+}\arrow{se}\node{M_1^-}\arrow{s}\\
\node{E_0^-}\arrow{ne,J}\arrow[2]{e,b,J}{\pi_1^+|_{E_0^-}}\node{}\node{M_1}
\end{diagram}
$$
which proves $\bl_{X}E_0^-\subset M_1^-$ completing the proof of part (c).
\end{proof}
We now study the third flip for $d\geq 7$ odd.  Since $2<\frac{d-1}{2}$, Remark \ref{hilbertwalls3} with $\ell=2$ and Proposition \ref{prop for remark} imply that the exceptional loci are:
\begin{eqnarray*}
E_2^+&:& 0\rightarrow \mathcal{I}_Z((d-3)/2)\rightarrow F \rightarrow\mathcal{I}_W^{\vee}((-d-3)/2)[1]\rightarrow 0\ \ \ |Z|=|W|=2\\
E_2^-&:& 0\rightarrow \mathcal{I}_W^{\vee}((-d-3)/2)[1]\rightarrow G \rightarrow \mathcal{I}_Z((d-3)/2) \rightarrow 0\ \ \ |Z|=|W|=2.
\end{eqnarray*}
Again, Bertram-Ein-Lazarsfeld vanishing exposes $E_2^+$ and $E_2^-$ as projective bundles over $\Hilb^2(\P^2)\times \Hilb^2(\P^2)$. 

Our plan is to study the restriction of the directed MMP for $M(0,d,-3d/2)$ to $E_0^-$. It is convenient to fix some notation. Inductively define $Y_1^+:=E_0^-$, $Y_i^-$ is the closure of the image of $Y_i^+$ by the rational map $M_i^+\dashrightarrow M_i^-$ and $Y^+_{i+1}:=Y_i^-$. Then, for instance, $Y_1^-=Y_2^+=\bl_XE_0^-$.
\begin{proposition} $E_2^+$ intersects $Y_2^+$ along the strict transform of the secant variety $\widetilde{Sec X}$ which is a projective bundle over $\Hilb^2(\P^2)$.
\end{proposition}
\begin{proof}
The computation is very similar to the one we did when computing $E_1^+\cap E_0^-$. Let $Z=p+q$ where $p,q\in \P^2$ and $p\neq q$. We have a pullback diagram
\begin{equation}\label{diagram for X2}
\begin{diagram}\dgARROWLENGTH=1.7em
\node{}\node{}\node{\mathcal{I}_Z\left((d-3)/2\right)}\arrow{sw}\arrow{s,J}\\
\node{\mathcal{O}\left((-d-3)/2\right)[1]}\arrow{e,J}\arrow{s,=}\node{G^{\bullet}_Z}\arrow{e,A}\arrow{s}\node{\mathcal{O}\left((d-3)/2\right)}\arrow{s,A}\\
\node{\mathcal{O}\left((-d-3)/2\right)[1]}\arrow{e,J}\node{\mathcal{G}}\arrow{e,A}\node{\C_Z}\\
\end{diagram}
\end{equation}
The difference here is that $\ext(1,\C_Z,\mathcal{O}((-d-3)/2)[1])=2$ which corresponds to the line passing through $p$ and $q$ removing $p$ and $q$. Thus, the intersection of $E_2^+\setminus E_1^-$ with $E_0^-\setminus X$ is $Sec X\setminus X$, which proves the first claim. 

A problem arises when considering $Z=2p$, because in this case we have $\ext(1,\C_Z,\mathcal{O}((-d-3)/2)[1])=3$. But since in $M_2^+$ we already flipped $E_1^+$ then not all the complexes $G_Z$ obtained this way are Bridgeland stable. Instead, the complexes $G_{2p}$ fitting into a commutative diagram
\begin{equation}
\begin{diagram}\dgARROWLENGTH=1.7em
\node{}\node{}\node{\mathcal{I}_{2p}\left((d-3)/2\right)}\arrow{sw}\arrow{s,J}\\
\node{\mathcal{I}_p((-d-3)/2)[1]}\arrow{e,J}\arrow{s,=}\node{G_{2p}}\arrow{e,A}\arrow{s}\node{\mathcal{I}_p((d-3)/2)}\arrow{s,A}\\
\node{\mathcal{I}_p((-d-3)/2)[1]}\arrow{e,J}\node{\mathcal{G}}\arrow{e,A}\node{\C_p}\\
\end{diagram}
\end{equation}
are Bridgeland stable. The objects $G_{2p}$ form the fiber of $\widetilde{Sec X}$ over $Z=2p$.
\end{proof}
\begin{lemma}\label{common blowup2}The fiber product $M_2^+\times_{M_2}M_2^-$ is isomorphic to the common blow-up $\bl_{E_2^+}M_2^+= \bl_{E_2^-}M_2^-$.
\end{lemma}
\begin{proof}
The proof is similar to the proof of Lemma \ref{common blowup}. The right vanishing is again a consequence of Bertram-Ein-Lazarsfeld vanishing.
\end{proof}
This completes Vermeire's first flip since by restricting the fiber diagram of Lemma \ref{common blowup2} one gets
$$
\begin{diagram}
\node{E}\arrow{e,J}\arrow{s}\node{\bl_{\widetilde{Sec X}}(\bl_{X}E_0^-)}\arrow{e,J}\arrow{s}\node{M_2^+\times_{M_2}M_2^-}\arrow{s}\arrow{e}\node{M_2^-}\arrow{s}\\
\node{\widetilde{SecX}}\arrow{seee}\arrow{e,J}\node{\bl_X(E_0^-)}\arrow{e,J}\node{M_2^+}\arrow{e}\node{M_2}\\
\node{}\node{}\node{}\node{\Hilb^2(\P^2)}\arrow{n,J}
\end{diagram}
$$
\begin{remark} In \cite{V1} it is mentioned that flips of secant varieties are closely related to the geometry of $\Hilb^n(X)$. By Propositions \ref{prop for remark} and \ref{MMP3} and the results of this section, one sees that indeed flips of secant varieties of Veronese surfaces are related to the geometry of $\Hilb^n(\P^2)$, and more precisely to its birational geometry.
\end{remark}  

By using diagrams similar to \eqref{diagram for X} and \eqref{diagram for X2} one sees that every rank-1 wall produces a birational transformation of $E_0^-$ whose exceptional locus contains the strict transform of some higher secant variety of $X$. Indeed, for $\ell <(d-1)/2$ the exceptional locus for the induced birational transformation of $E_0^-$, corresponding to crossing the wall $W_{\ell}$, is the strict transform of $Sec^{\ell-1}X$. For $\ell\geq (d-1)/2$ the exceptional locus is reducible and the middle component $E_{\ell,0}$ intersects $E_0^-$ along the strict transform of $Sec^{\ell-1}X$. The intersection $E_{\ell,-i}\cap E_{\ell,0}\cap E_0^-$ is the locus in $\widetilde{Sec^{\ell-1}X}$ of $(\ell-1)$-dimensional planes passing through $\ell$ different points, $i(d-i)/2$ of them lying on the image by the $(d-3)$-uple embedding of a curve $C\subset \P^2$ of degree $i$. Since $E_0^-$ is fixed by the duality automorphism, this completely describes the loci for the restriction to $E_0^-$ of the MMP on $M_H(0,d,-3d/2)$.  

%%%%%%%%%%%%%%%%%%%%%%%%%%%%% 
\subsection{The divisorial contraction}
We want to study what happens to our restricted MMP when crossing the wall $W_{ch(\mathcal{O}),v}$ corresponding to the theta divisor (i.e., the closure of the set of those sheaves that admit at least one nonzero section). From Lemma \ref{lemma dual} it follows that the duality automorphism preserves the wall-crossing, and so the theta divisor is left invariant by duality. Therefore, it corresponds to extensions of the form
$$
0\rightarrow N \rightarrow F\rightarrow \mathcal{O}(-3)[1]\rightarrow 0
$$
where $N$ is an object in $\mathcal{A}_{-3/2}$ of Chern character 
$$
ch(N)=\left(1, d-3, \frac{9}{2}-\frac{3d}{2}\right)=\left(1, d-3, \frac{(d-3)^2}{2}-\frac{d(d-3)}{2}\right).
$$
At the wall $W_{ch(\mathcal{O}),v}=W_{ch(N),v}$, $N$ must be Bridgeland semistable. The moduli space of Bridgeland semistable objects at the wall $W_{ch(N),v}$ of Chern character $ch(N)$ is birational to the Hilbert scheme parametrizing (twisted) ideal sheaves $\mathcal{I}_Z(d-3)$ of 0-dimensional subschemes $Z\subset \P^2$ of length $n=d(d-3)/2$.  
  
\begin{remark}
One can originally think of the dual extensions but this version allows us to compute the intersection with the first flipped locus more effectively.
\end{remark}

The intersection of the theta divisor with $E_0^-$ corresponds to the extensions $F$ fitting into the push-forward diagrams
$$
\begin{diagram}
\node{\mathcal{O}_C(-3)}\arrow{e,J}\arrow{s,J}\node{G}\arrow{e,A}\arrow{s}\node{\mathcal{O}((d-3)/2)}\arrow{s,=}\\
\node{\mathcal{O}((-d-3)/2)[1]}\arrow{e,J}\arrow{s,A}\node{F}\arrow{sw,A}\arrow{e,A}\node{\mathcal{O}((d-3)/2)}\\
\node{\mathcal{O}(-3)[1]}\node{}\node{}
\end{diagram}
$$ 
where $C\subset \P^2$ is a curve of degree $(d-3)/2$. Indeed, the middle vertical sequence of arrows is exact and $G$ corresponds to those complexes produced when flipping the locus in $\Hilb^n(\P^2)$ of $n$ points on a curve of degree $(d-3)/2$. This intersection is therefore a projective bundle over the Hilbert scheme of plane curves of degree $(d-3)/2$. An example of this situation was already observed in \cite[Section 6]{BMW} for the case $d=5$  where the intersection of the theta divisor with $E_0^-$ was exactly the strict transform of the secant variety of the Veronese surface in $\P^5$. 

Notice that this intersection is not exactly what gets contracted when crossing $W_{ch(\mathcal{O}),v}$ since after several flips we may have replaced some of these objects by new ones. What we know is that because such $F$ is fixed by the duality automorphism then $F$ must have $\mathcal{O}$ as a subobject, and since $\Hom(\mathcal{O},\mathcal{O}(-3)[1])=0$ then the object $G$ above must have $\mathcal{O}$ as a subobject. Therefore crossing $W_{ch(\mathcal{O}),v}$ must introduce objects $E$ fitting into an exact sequence
$$
0\rightarrow \mathcal{O}(-3)[1]\rightarrow E\rightarrow G\rightarrow 0
$$  
with $G$ being $S$-equivalent to $\mathcal{O}\oplus \mathcal{G}$ at the wall $W_{ch(\mathcal{O}),v}$. Notice that $ch(\mathcal{G})=(0,d-3,-3(d-3)/2)$. These new objects $E$ are all strictly Bridgeland semistable, in fact pseudo-stable when $\mathcal{G}$ is pseudo-stable. 
\begin{remark}\label{last birational model} 
Because of the correspondence between Bridgeland wall-crossing and MMP for the moduli spaces $M_H(0,d,-3d/2)$, we know that after the divisorial contraction all the moduli spaces of Bridgeland semistable objects are isomorphic until they get finally contracted at the collapsing wall. This says that to understand the last birational model, it is enough to understand the moduli spaces at the wall $W_{ch(\mathcal{O}),v}$. 
\end{remark}
\begin{remark}
After a more detailed analysis one can further prove that if $\mathcal{G}$ is a sheaf then it must fit into an exact sequence
$$
0\rightarrow \mathcal{O}_C(-3)\rightarrow \mathcal{G}\rightarrow \mathcal{O}_{C}((d-3)/2)\rightarrow 0,
$$
and if $\mathcal{G}$ is a complex then (using the semistability of $\mathcal{G}$ at $W_{ch(\mathcal{O}),v}$) that at least it has to fit into an exact sequence of the form
$$
0\rightarrow A\rightarrow \mathcal{G}\rightarrow A^{D}\rightarrow 0
$$
where $A$ is a semistable object of invariants $\displaystyle ch(A)=\left(0,\frac{d-3}{2},-\frac{(d-3)(d+9)}{8}\right)$.
\end{remark}
\subsection{The last birational model} One could ask what is the moduli space we obtain after the divisorial contraction and what is the strict transform of $E^-_0$. The answer to the first question comes from the identification with the quiver moduli. Values of $(-3/2,t)$ near $W_{ch(\mathcal{O}(-1)),v}$ are all inside the quiver region 
$$
(s-2)^2+t^2<1
$$
corresponding to $k=-1$. Recall from Theorem \ref{quiver}, that for every $(s,t)$ in this region, there is a vector $\mathfrak{a}_{s,t}$ orthogonal to the dimension vector
$$
\mathfrak{n}=B_{-1}v=\begin{bmatrix}
\frac{(-1)(-1-1)}{2} & \frac{-(2(-1)-1)}{2} & 1\\
(-1)(-1-2) & -(2(-1)-2) & 2\\
\frac{(-1-1)(-1-2)}{2} & \frac{-(2(-1)-3)}{2}& 1
\end{bmatrix}
\begin{bmatrix}
0\\
d\\
-3d/2
\end{bmatrix}=
\begin{bmatrix}
0\\
d\\
d
\end{bmatrix}
$$
such that the moduli space of $\sigma_{s,t}$-semistable objects of type $v=(0,d,-3d/2)$ can be identified with the moduli space of complexes
$$
\C^{d}\otimes \mathcal{O}(-2)\rightarrow \C^{d}\otimes \mathcal{O}(-1)
$$
that are semistable for the vector $\mathfrak{a}_{s,t}$. Note that $\mathfrak{a}_{s,t}$ is of the form $(a,-\theta,\theta)$. Since one naturally has the subcomplex 
$$
\begin{diagram}
\node{\C^{d}\otimes \mathcal{O}(-2)}\arrow{e}\node{\C^{d}\otimes \mathcal{O}(-1)}\\
\node{}\node{\C^{d}\otimes \mathcal{O}(-1)}\arrow{n,J}
\end{diagram}
$$
then one must have $(0,0,d)\cdot (a,-\theta,\theta)=\theta\geq 0$ whenever there is a quiver semistable complex of dimension vector $(0,d,d)$ with respect to $\mathfrak{a}_{s,t}$. Moreover, the proof of \cite[Proposition 8.1]{ABCH} shows that above $W_{ch(\mathcal{O}(-1)),v}$ we have $\theta>0$, at $W_{ch(\mathcal{O}(-1)),v}$ we have $\theta=0$, and below $W_{ch(\mathcal{O}(-1)),v}$ we have $\theta<0$. 

\begin{remark}
In \cite{K}, King shows that the GIT quotient
$$
\Hom(W\otimes\mathcal{O}(-2),W^*\otimes\mathcal{O}(-1))//GL(W)\times GL(W^*),
$$
where the action is given by conjugation, is the moduli space of quiver semistable complexes of dimension vector $\mathfrak{n}=(0,d,d)$, with respect to the orthogonal vector $\mathfrak{a}=(a,\theta,\theta)$ . Different choices of the linearization for the GIT quotient correspond to taking $\theta>0$, $\theta=0$, or $\theta<0$.
\end{remark}

Therefore the last model corresponds to the moduli space $N(3,d,d)$ studied in \cite{DM} of morphisms 
$$
W\otimes\mathcal{O}(-2)\rightarrow W^*\otimes\mathcal{O}(-1),
$$
that are GIT semistable with respect to the natural action of $GL(W)\times GL(W^*)$, where $\dim W=d$. The moduli space at $W_{ch(\mathcal{O}(-1)),v}$ is just a point, and below $W_{ch(\mathcal{O}(-1)),v}$ the moduli space is empty, proving our assertion that $W_{ch(\mathcal{O}(-1)),v}$ was the collapsing wall.

In order to understand what is going to be the last birational model of $E_0^-$, let us take a look at the simplest but yet interesting $M_H(0,3,-9/2)$ studied by Le Potier in \cite{LP2}. In \cite[Th\'eor\`eme 4.4 and Lemme 4.5]{LP2}, Le Potier showed that $M_H(0,3,-9/2)$ is the blow up of $N(3,3,3)$ at the complement of the dense open subset of injective morphisms $3\mathcal{O}(-2)\hookrightarrow 3\mathcal{O}(-1)$, this complement consists of a single $GL(3)\times GL(3)$-orbit, which is the orbit of the skew-symmetric matrix 
$$
\begin{pmatrix}0& -z& x\\ z&0&-y\\ -x&y&0\end{pmatrix}.
$$
As a complex, its cohomology is represented in the diagram
$$
\begin{diagram}\dgARROWLENGTH=1.7em
\node{\mathcal{O}(-3)}\arrow{se,J}\node{}\node{}\node{}\node{\mathcal{O}}\\
\node{}\node{3\mathcal{O}(-2)}\arrow[2]{e,t}{\tiny{\begin{pmatrix}0& -z& x\\ z&0&-y\\ -x&y&0\end{pmatrix}}}\arrow{se,A}\node{}\node{3\mathcal{O}(-1)}\arrow{ne,A}\node{}\\
\node{}\node{}\node{\Omega^1}\arrow{ne,J}\node{}\node{}
\end{diagram}
$$
where the diagonal exact sequences are the Euler sequences. 

The example above reflects some general features of the general situation. Note that a complex $W\otimes\mathcal{O}(-2)\rightarrow W^*\otimes\mathcal{O}(-1)$ given by a skew map is fixed by the duality automorphism, since it corresponds to taking the negative transpose of the corresponding matrix. The general skew map will drop rank by 1 everywhere and therefore it must have a kernel and a cokernel that are line bundles. A simple computation of the invariants shows that the kernel should be $\mathcal{O}((-d-3)/2)$ and its cokernel $\mathcal{O}((d-3)/2)$, i.e., as a complex it should fit into an exact sequence
$$
0\rightarrow \mathcal{O}((-d-3)/2)[1]\rightarrow [W\otimes\mathcal{O}(-2)\rightarrow W^*\otimes\mathcal{O}(-1)]\rightarrow \mathcal{O}((d-3)/2)\rightarrow 0
$$  
in $\mathcal{A}_{-3/2}$, giving a point in $E_0^{-}$. Conversely, by Proposition \ref{pseudostableisstable} all the complexes in $E_0^{-}$ are stable rather than pseudo-stable. Any stable complex in the last model for $E_0^{-}$ must be, as $E_0^-$ itself, fixed by the duality automorphism and therefore it must correspond to the orbit of a skew-map $W\otimes\mathcal{O}(-2)\rightarrow W^*\otimes\mathcal{O}(-1)$.

\begin{remark}
For $d=5$ we have four walls (see \cite[Section 6.2]{BMW} for details): the walls produced by the destabilizing objects $\mathcal{O}(1)$ and $\mathcal{I}_p(1)$, $W_{ch(\mathcal{O}),v}$, and $W_{ch(\mathcal{O}(-1)),v}$. At the first two walls the Jordan-H\"older filtrations have length 1 and so the strict transform of $E_0^-$ before the divisorial contraction consists only of stable objects. As above, it can be proven that at the wall $W_{ch(\mathcal{O}),v}$, the divisorial contraction produces objects that are $S$-equivalent to complexes fitting into an exact sequence
$$
0\rightarrow \mathcal{O}(-3)[1]\rightarrow E\rightarrow \mathcal{O}\oplus \mathcal{O}_{\ell}(-3)\oplus\mathcal{O}_{\ell}(1)\rightarrow 0.
$$

In $N(3,5,5)$ these correspond to the $GL(W)\times GL(W^*)$-orbits of matrices
$$
\begin{pmatrix}
0& -z& x& 0 & 0 \\
 z&0&-y& 0 & 0 \\
 -x&y&0& 0 & 0 \\
 0 & 0 & 0&0 &L\\
 0 & 0 & 0&-L&0
\end{pmatrix}
$$
where $L$ is a linear equation defining $\ell$. The $GL(W)\times GL(W^*)$-orbits of these matrices are strictly semistable in $N(3,5,5)$. 
\end{remark}
%For $d>5$ the orbits generated after the divisorial contraction are the orbits of elements of the form
%$$
%\begin{pmatrix}
%0& -z& x& 0 & 0 \\
 %z&0&-y& 0 & 0 \\
 %-x&y&0& 0 & 0 \\
 %0 & 0 & 0&0 &A\\
 %0 & 0 & 0&-A^t&0
%\end{pmatrix}
%$$
%where $A$ is a square matrix of order $(d-3)/2$ but we can not say much about it, only that when $\mathcal{G}$ is a sheaf then $\det A$ must give an equation for the curve $C$.

Now assume that $B$ is a skew-symmetric matrix giving a $GL(W)\times GL(W^*)$-stable orbit. If there are invertible matrices $T,S\in GL(W)$ such that $TBS^t$ is again skew-symmetric then
$$
B=(S^{-1}T)B(T^{-1}S)^t
$$ 
and therefore $S=\lambda T$ for some $\lambda\in \C^*$ since $B$ is stable and so $\Hom(B,B)=\C$. 

Since $GL(W)$ can be embedded via the diagonal $T\mapsto (T,T^t)$ into $GL(W)\times GL(W^*)$ then a skew-symmetric matrix that is $GL(W)\times GL(W^*)$-stable is also $GL(W)$-stable for the diagonal action. Thus we have an injective map 
$$
\{\text{Stable Skew}\ GL(W)\times GL(W^*)-\text{orbits}\}\longrightarrow \wedge^2W\otimes V//GL(W)
$$
where $V=\Hom(\mathcal{O}(-2),\mathcal{O}(-1))$, and the action of $GL(W)$ on $\wedge^2W\otimes V$ is the natural one:
$$
GL(W)\times \wedge^2W\otimes V\rightarrow \wedge^2W\otimes V,\ \ \ (S,B)\mapsto SBS^t.
$$
In the examples above, this map can be extended to the semistable orbits that have a $skew$ representative. In fact, in a personal communication to the author, Aaron Bertram has made the following
\begin{conjecture} The last birational model of $\bl_{X}E_0^-$ is isomorphic to the GIT quotient $\wedge^2W\otimes V//GL(W)$.
\end{conjecture}

\subsection{Odd Veronese embeddings}
%\vspace{12pt}

As mentioned at the beginning of Section 5, one could as well study flips for secant varieties of odd Veronese embeddings by studying the Bridgeland wall crossing for the Gieseker moduli space $M_H(0,d,-d)$ for $d$ even, which contains a locus parametrizing curves of degree $d$. As before, we want to run the MMP on $M_H(u)$ for $u=(0,d,-d)$. 

Let $E$ be a Gieseker semistable sheaf with $ch(E)=u$. The center and the radius of a wall produced by a Bridgeland destabilizing subobject $A\hookrightarrow E$ are respectively
\begin{align*}
C&=\left(\frac{ch_2(E)}{ch_1(E)},0\right)=(-1,0),\\
R&=\sqrt{\left(\frac{ch_2(E)}{ch_1(E)}\right)^2+2\frac{ch_2(A)}{ch_0(A)}-2\frac{ch_2(E)ch_1(A)}{ch_0(A)ch_1(E)}}
&=\sqrt{1+2\frac{ch_1(A)+ch_2(A)}{ch_0(A)}}.
\end{align*}
Thus the category to be considered is $\mathcal{A}_{-1}$. Let $E$ be a Gieseker semistable sheaf with $ch(E)=u$, then
$$
\chi(E)=\dim H^0(E)-\dim H^1(E)=ch_2(E)+\frac{3}{2}ch_1(E)=-d+\frac{3d}{2}=\frac{d}{2}>0.
$$
Thus there are nonzero maps $\mathcal{O}\rightarrow E$. If $K$ is the kernel in $\mathcal{A}_{-1}$ of such map, then there is a diagram in $\mathcal{A}_{-1}$
$$
\begin{diagram}
\node{K}\arrow{e,J}\node{\mathcal{O}}\arrow{se,A}\arrow{e}\node{E}\\
\node{}\node{}\node{F}\arrow{n,J}
\end{diagram}
$$
Because $\mathcal{O}$ and $E$ are sheaves so are $K$ and $F$. Thus $K$ must be a subsheaf of $\mathcal{O}$ implying that $F$ has rank 0, and therefore is a subsheaf of $E$. Now, as mentioned before $\mathcal{O}$ is $\sigma_{s,t}$-stable for all $s,t$ ($t>0$), and so
$$
\mu_{s,t}(K)<\mu_{s,t}(\mathcal{O})<\mu_{s,t}(F)\leq \mu_{s,t}(E)
$$ 
unless $K$ is trivial. Since such inequality does not hold for all $s,t$, then $\mathcal{O}$ is a subobject of $E$ in $\mathcal{A}_{-1}$. This 
proves that we have a collapsing wall $W_{ch(\mathcal{O}),u}$, which has radius
$$
R=\sqrt{1+2\frac{ch_1(\mathcal{O})+ch_2(\mathcal{O})}{ch_0(\mathcal{O})}}=1.
$$
Crossing this wall collapses (at least) the open set of Gieseker semistable sheaves that are Bridgeland semistable along $W_{ch(\mathcal{O}),u}$.

There is also a divisorial contraction produced by the tangent sheaf $T_{\P^2}(-1)$. To see this, consider the rational map 
$$
\xymatrix{M_H(0,d,-d) \ar@{-->}[r] & M_H(0,2d,-3d)},\ \ \mathcal{F}\mapsto \mathcal{F}\otimes \Omega^1(1).
$$ 
As in the case when $d$ is odd, the moduli space $M_H(0,2d,-3d)$ has a natural divisor $\Theta$ (sheaves with a section), and the image of $M_H(0,d,-d)$ is not contained in $\Theta$. By pulling back $\Theta$ we obtain a divisor $\Theta'$ consisting of semistable sheaves $\mathcal{F}$ such that $\mathcal{F}\otimes\Omega^1(1)$ has a section. Since 
$$
\mbox{Hom}(\mathcal{O},\mathcal{F}\otimes\Omega^1(1))=\mbox{Hom}(T_{\P^2}(-1),\mathcal{F})
$$ 
then the divisor $\Theta'$ is contracted when crossing the wall corresponding to the destabilizing object $T_{\P^2}(-1)$. The wall $W_{ch(T_{\P^2}(-1)),u}$ has radius
$$
R=\sqrt{1+2\frac{ch_1(T_{\P^2}(-1))+ch_2(T_{\P^2}(-1))}{ch_0(T_{\P^2}(-1))}}=\sqrt{1+2\frac{1-1/2}{2}}=\sqrt{\frac{3}{2}}.
$$
\begin{remark}
To analyze the last birational model one has to use the triad $\mathcal{O}(-1),\Omega^1(1),\mathcal{O}$ instead of $\mathcal{O}(-2),\mathcal{O}(-1),\mathcal{O}$ in the construction of the quiver moduli. This gives a construction of the last birational model as the GIT quotient
$$
\Hom(V\otimes \mathcal{O}(-2),W\otimes \mathcal{O})//GL(V)\times GL(W),
$$
where $V$ and $W$ are complex vector spaces of dimension $d/2$.
\end{remark}
From the inequalities in the previous chapter, one obtains that the invariants of a rank 1 destabilizing subobject producing a wall corresponding to a flip must satisfy
$$
\frac{3}{2}< 1+2ch_1(A)+2ch_2(A)\leq \frac{d^2}{4}\ \ \text{and}\ \ R\leq ch_1(A)+1\leq d-R,
$$ 
where 
$$
R=\sqrt{1+2ch_1(A)+2ch_2(A)}
$$ 
is the radius of the corresponding wall. 

Notice that $ch_1(A)=(d-2)/2$ always satisfies the second inequality. Trying to mimic what we did for the $d$ odd case, we would like $ch_1(A)=(d-2)/2$ to be the only solution to the second inequality, i.e., looking for destabilizing objects of the form
$$
A=\mathcal{I}_Z\left(\frac{d-2}{2}\right),\ \ \text{length}(Z)=\ell.
$$
If this is the only solution to the second inequality, it means that $(d-R)-R<1$, equivalently we must have
$$
\frac{(d-1)^2}{4}<1+2ch_1(A)+2ch_2(A).
$$ 
After some elementary computations we obtain
$$
0\leq \ell \leq\frac{d-2}{4}.
$$
\begin{remark}
The range $0\leq \ell \leq\frac{d-2}{4}$ is only to assure that the exceptional locus for the flips is irreducible, and to obtain good wall-crossing since at these walls the strictly semistable objects being destabilized will have Jordan-H\"older filtrations of length 1. 
\end{remark}
The exceptional locus for the first flip corresponds to those sheaves fitting into an exact sequence in $\mathcal{A}_{-1}$ of the form
$$
0\rightarrow \mathcal{O}((d-2)/2)\rightarrow E \rightarrow \mathcal{O}((-d-2)/2)[1]\rightarrow 0.
$$
The exceptional introduced when crossing this wall is 
$$
\P(\mbox{Ext}^{1}(\mathcal{O}((d-2)/2),\mathcal{O}((-d-2)/2)[1]))=\P(H^0(\mathcal{O}(d-3)))^{\vee}.
$$
Propositions 5.2, 5.3, 6.1, Lemma 6.3, Propositions 6.4, 6.5, and Lemma 6.6 hold in this setting with identical proofs after replacing $(d-3)/2$ in the odd case by $(d-2)/2$ in the even case, and noticing that the correct duality automorphism is $(\cdot)^{D}\otimes \mathcal{O}(1)$ instead of $(\cdot)^D$. 

We conclude this chapter with the following theorem, which is a corollary of our construction:
\begin{theorem}\label{main3}
Let $d\geq 5$ be an integer and let $\nu_{d-3}\colon \P^2\rightarrow \P(H^0(\mathcal{O}(d-3)))^{\vee}=\P^N$ be $(d-3)$-uple embedding. There exists a sequence of flips 
\footnotesize{
$$
\begin{diagram}
\node{\bl_{\nu_{d-3}(\P^2)}\P^N}\arrow{se}\arrow{s}\arrow[2]{e,t,..}{f_1}\node{}\node{M_1}\arrow[2]{e,t,..}{f_2}\arrow{sw}\arrow{se}\node{}\node{M_2}\arrow{sw}\arrow{se}\node{\cdots}\node{M_{k}}\arrow{sw}\\
\node{N_1\supset\P^N\hspace{1cm}}\node{M_1'}\node{}\node{M_2'}\node{}\node{\cdots}\node{}
\end{diagram} 
$$}
\normalsize
where $k=(d-1)/2$ for $d$ odd, and $k=\lfloor{(d-2)/4\rfloor}$ for $d$ even, the exceptional locus of $f_i$ is the strict transform of $Sec^{i}(\nu_{d-3}(\P^2))$, and $N_1$ is the first birational model appearing when running the MMP for $M_H(0,d,-3d/2)$ or $M_H(0,d,-d)$ depending on whether $d$ is odd or even respectively.
%$$
%\begin{diagram}
%\node{\bl_{\nu_{d-3}(\P^2)}\P^N}\arrow{se}\arrow{s}\arrow[2]{e,t,..}{f_1}\node{}\node{M_1}\arrow{sw}\arrow{se}\node{\cdots}\node{M_{\ulcorner\frac{d-1}{2}\urcorner}}\arrow{sw}\\
%\node{N_1\supset\P^N\hspace{1cm}}\node{M_1'}\node{}\node{\cdots}\node{}
%\end{diagram} 
%$$
%where $Ex(f_i)$ is the strict transform of $Sec^{i}(\nu_{d-3}(\P^2))$ and $N_1$ is the first birational model of $N(d,0)$ or $N(d,d/2)$ depending on whether $d$ is odd or even.
\end{theorem}
\begin{remark} As we have seen, this sequence of flips is indeed longer but the exceptional loci after the first $k$ flips become more complicated since after this point strictly semistable objects have at least three Jordan-H\"older factors.
\end{remark}

%%%%%%%%%%%%%%%%%%%%%%%%%%%%%%%%%%%%%%%%%%%%%%%%%%

\section*{Acknowledgments}
The present work is part of the author's Ph.D. thesis at the University of Utah. He would like to thank his advisor, Professor Aaron Bertram, for his vision, and constant encouragement and patience throughout the years of graduate school. The author would like to express his gratitude to the referee(s) for so many valuable comments that have made the present work much more readable.

%\begin{thebibliography}{0}
%\bibitem{1} R. Lorentz and D. B. Benson, Deterministic 
%and nondeterministic flow-chart interpretations, 
%{\it J. Comput. System Sci}.  {\bf 27} (1983) 400--433.

%\bibitem{2} M. J. Beeson, {\it Foundations of Constructive 
%Mathematics} (Springer, Berlin, 1985).

%\bibitem{3} K. L. Clark, Negations as failure, in 
%{\it Logic and Data Bases}, eds. H. Gallaire and\break 
%J. Winker (Plenum Press, New York, 1973), pp.~293--306.

%\bibitem{4} M. Joliat, A simple technique for partial 
%elimination of unit productions from LR(k) parsers, 
%{\it IEEE Trans. Comput}. {\bf C-27} (1976) 753--764.

%\bibitem{5} D. Dolve, Unanimity in an unknown and 
%unreliable environment, in {\it Proc. 22nd Annual Symposium on 
%Foundations of Computer Science}, Nashville, TN (Oct. 1981),\break 
%pp.~159--168.

%\bibitem{6} R. Tamassia, C. Batini and M. Talamo, An  
%algorithm for automatic layout of entity relationship diagrams, 
%in {\it Entity-Relationship Approach to Software  
%Engineering}, {\it Proc.  3rd Int. Conf. on Entity-Relationship 
%Approach}, eds. C. G. Davis, S.  Jajodia, P. A. Ng and R. T. Yeh 
%(North-Holland, Amsterdam, 1983), pp.~421--439.

%\bibitem{7} W. L. Gewirtz, Investigations in the theory
%of descriptive complexity, Ph. D. Thesis, New York University
%(1974).
%\end{thebibliography}

\end{document}